\documentclass[letterpaper,12pt]{article}
\usepackage{braket}
\usepackage{diagbox}
\usepackage{pdflscape} 
\usepackage{adjustbox}

\usepackage{tabularx} % extra features for tabular environment
\usepackage{amsmath,amsthm,amssymb}  % improve math presentation
\usepackage{amsmath}
\usepackage{amsfonts}
\usepackage{amssymb}
\usepackage{graphicx} % takes care of graphic including machinery
\usepackage[margin=0.8in,letterpaper]{geometry} % decreases margins
\usepackage{cite} % takes care of citations
\usepackage[english]{babel}
\usepackage{caption}
\usepackage{float}
\usepackage[table]{xcolor}
\usepackage{array}
\usepackage{algorithm}
\usepackage{algpseudocode}
\usepackage{cases}
\usepackage{titlesec}
\usepackage{tikz}
\usepackage[utf8]{inputenc}
\usepackage[T1]{fontenc}
\usepackage{bm}
\usepackage{mathtools}
\usepackage{extarrows}
\usepackage{booktabs} 
\usepackage{dsfont}
\usepackage{authblk}
\usepackage{multirow}
\usepackage{hyperref}
\usepackage{parskip}
\usepackage{pgfplots}
\pgfplotsset{compat=1.18}
\usepackage{rotating}
\usepackage{csquotes}

\newtheorem{theorem}{Theorem}
\newtheorem{lemma}{Lemma}
\newtheorem{proposition}{Proposition}
\newtheorem{corollary}{Corollary}
\newtheorem{definition}{Definition}
\newtheorem{example}{Example}

\newtheorem{remark}{Remark}

\newtheorem{conjecture}{Conjecture}
\newtheorem{observation}{Observation}

\usepackage{microtype}      % microtypography
\usepackage{lipsum}
\usepackage{nicefrac}
\usepackage{fancyhdr}       % header

 \title{The Power Contamination Problem on Grids Revisited: Optimality, Combinatorics, and Links to Integer Sequences}

 \author{
El-Mehdi Mehiri$^{a,b}$\footnote{Corresponding author. Email: elmehdi.mehiri@emse.fr or  mehiri314@gmail.com}  and Mohammed L. Nadji$^{b}$\footnote{Email: m.nadji@usthb.dz}
}
\date{}

\begin{document}

\maketitle
\begin{center}
{\small
$^{a}$ Mines Saint-Étienne, CMP, Department of Manufacturing Sciences and Logistics, 
F-13120 Gardanne, France. \\[4pt]
$^{b}$ University of Science and Technology Houari Boumediene, Faculty of Mathematics, RECITS Laboratory, Po. Box 32, El-Alia,
16111 Bab-Ezzouar, Algiers, Algeria.
}
\end{center}

\begin{abstract}
\noindent We study the power contamination problem on rectangular grid graphs. Given an initial set of contaminated cells, contamination propagates according to local geometric rules, and the objective is to minimize the size of a set that contaminates the whole grid. We disprove the previously conjectured formula for the contamination number and determine its exact value for every grid \(G(n,m)\). We also derive recurrence relations and structural properties of optimal contamination sets. In addition, we initiate the enumeration of optimal solutions and obtain explicit formulas for several grid families. These counting results reveal connections with classical combinatorial structures, including ternary words with forbidden factors and permutations, and suggest a further link with the large Schr\"oder numbers. The paper therefore provides both an exact optimization result and a first systematic combinatorial analysis of optimal contamination configurations on grids.
\end{abstract}

\medskip
\noindent\textbf{Keywords:} Power contamination; power domination; grid graphs; combinatorial optimization; optimal configurations; integer sequences.

\medskip
\noindent\textbf{MSC (2020):} 05C69; 05C57; 05A15; 90C27; 05C90.
\section{Introduction}

Power domination is a propagation-based variant of domination that arose from the problem of monitoring electric power networks by placing phase measurement units (PMUs) at selected nodes   \cite{BrueniHeath2005}. In the graph-theoretic model introduced by Haynes et al. \cite{HaynesHedetniemiHedetniemiHenning2002}, a set of initially monitored vertices first observes its closed neighborhood, after which further vertices become observed through an iterative propagation rule reflecting the inference mechanisms available in electrical networks. This viewpoint, together with the simplified formulation of Brueni and Heath \cite{BrueniHeath2005}, has made power domination a well-established topic at the interface of graph theory, combinatorial optimization, and network monitoring; see, for example, \cite{ZhaoKangChang2006,Liao2016,FerreroHogbenKenterYoung2017,WhitlatchShultisRamirezOrtizKniahnitskaya2023}. Like the original domination set problem \cite{Haynes1998}, the power domination problem has been proven to be NP-complete \cite{Guo_2008}.

From the algorithmic and structural points of view, classical power domination has been studied on a wide variety of graph classes. Exact values, upper bounds, and efficient algorithms are known for several families, including block graphs, generalized Petersen graphs, permutation graphs, maximal planar graphs, regular claw-free graphs, rectangular grids, triangular grids with triangular or hexagonal boundary, Kn\"odel graphs, Hanoi graphs, various chemical and network-inspired structures, and Mycielskian constructions; see, among others, \cite{Dorfling2006,Ferrero2016,XuKangShanZhao2006,XuKang2011,Wilson2019,DorbecGonzalezPennarun2019,LuMaoWang2020,BoseGledelPennarunVerdonschot2020,StephenRajanRyanGrigoriousWilliam2015,VargheseVijayakumarHinz2018,SundaraRajanArulanandPrabhuRajasingh2023,SreethuKVargheseVarghese2024}. These results illustrate a recurring phenomenon in the area: although the general problem is computationally difficult, many structured families admit precise combinatorial descriptions.

Another important direction concerns the behavior of power domination under graph operations, especially graph products. Dorbec, Mollard, Klav\v zar, and \v Spacapan studied power domination in product graphs and obtained exact or near-exact results for direct, strong, and lexicographic products, in particular for products of paths \cite{DorbecMollardKlavzarSpacapan2008}. Later work focused more specifically on Cartesian products, providing additional constructions, bounds, and Vizing-like inequalities \cite{KohSoh2016,AndersonKuenzel2022}. This perspective is especially relevant for grid-like graphs, since many natural lattice families can be viewed as products or product-like constructions.

For the present paper, however, the literature on \emph{variants} of power domination is particularly relevant. A major generalization is \(k\)-power domination, introduced by Chang, Dorbec, Montassier, and Raspaud, which interpolates between domination and classical power domination and has since been developed for hypergraphs, block graphs, weighted trees, Sierpiński  graphs,  and regular claw-free graphs \cite{Dorbec2014, ChangDorbecMontassierRaspaud2012,ChangRoussel2015,WangChenLu2016,ChengLuZhou2020,ChenLuYe2022}. Other variants modify the admissible initial sets, the optimization criterion, or the robustness requirements, leading to restricted power domination, fault-tolerant power domination, connected power domination, failed power domination, throttling, and random-failure models \cite{PaiChangWang2010,BozemanBrimkovEricksonFerreroFlaggHogben2019,BrimkovMikesellSmith2019,GlasserJacobLedermanRadziszowski2020,BrimkovCarlsonHicksPatelSmith2019,GanesamurthySrimathiJeyaranjani2025,BjorkmanBrennanFlaggKoch2026}. Together, these works show that even small changes in the monitoring process can lead to mathematically distinct problems with new combinatorial and algorithmic features.

Beyond optimization of minimum-size monitoring sets, several more global aspects of the process have also been investigated. Zhao, Kang, and Chang established general upper bounds for the power domination number in connected graphs and connected claw-free cubic graphs \cite{ZhaoKangChang2006}. Liao introduced a bounded-time version of the problem, where the propagation phase is restricted to a prescribed number of rounds \cite{Liao2016}, while Ferrero, Hogben, Kenter, and Young formalized the notion of power propagation time and studied its interaction with lower bounds on the power domination number \cite{FerreroHogbenKenterYoung2017}. More recently, Whitlatch et al.\ initiated the enumerative study of power dominating sets by counting power dominating sets of fixed size in complete \(m\)-ary trees and relating this count to probabilities of successful random placement \cite{WhitlatchShultisRamirezOrtizKniahnitskaya2023}. These developments indicate that the area now extends well beyond the determination of a single minimum-cardinality parameter.

The present paper belongs to this broader line of work, but studies a different propagation model. More precisely, we consider the \emph{power contamination problem} on rectangular grids, introduced by Ainouche and Bouroubi as a dynamic variant related to power domination \cite{Ainouche2021}. In this model, propagation is no longer governed by the standard unique-neighbor forcing rule. Instead, contamination is triggered by specific local geometric configurations involving Moore and von Neumann neighborhoods. As a consequence, the problem is inherently geometric, and the structure of feasible initial configurations is strongly constrained by the shape of the grid.

In this paper, we revisit the power contamination problem on grids from a combinatorial point of view. We first disprove the conjecture proposed in \cite{Ainouche2021} for the contamination number of rectangular grids, and then determine the exact value of this parameter for all grids \(G(n,m)\). We also derive recurrence relations and investigate the enumeration and structure of optimal contamination sets. This leads to several connections with classical integer sequences and pattern-avoidance phenomena, showing that the model gives rise not only to an optimization problem on grids, but also to a rich enumerative combinatorial framework.

The paper is organized as follows. In Section~\ref{sec:definitions}, we introduce the model and the terminology used throughout the paper. Section~\ref{sec:structural} establishes structural properties of the contamination process on grids. In Section~\ref{sec:value}, we determine the exact value of the power contamination number. Section~\ref{sec:enumeration} is devoted to the enumeration and structure of optimal contamination sets, where several combinatorial connections emerge. Finally, Section~\ref{sec:conclusion} summarizes the main results and discusses further perspectives.

\section{Definitions and Preliminaries}
\label{sec:definitions}
In this section, we introduce the power contamination model on rectangular grids and fix the terminology used throughout the paper. We first describe the contamination rules and the associated propagation process, and then present several notions and special configurations that will be used in the proofs of the main results.

\subsection{The contamination model}
We begin by defining the grid \(G(n,m)\), the relevant neighborhoods of a cell, and the local rules governing contamination. These rules determine the propagation process from an initial contamination set and form the basis of the power contamination problem studied in this paper.

Let $G(n,m)=(V,E)$ be a grid graph, where the integers $n\geq 1$ and $m\geq 1$ represent the number of rows and columns, respectively. To avoid symmetry, we consider only horizontal grids, i.e.,  $m\geq n\geq 1$. 
Each cell is identified by its coordinates $(i,j)$, where \(i \in [n]\) and \(j \in [m]\) denote the row and column indices, respectively, with \([k]=\{1,\ldots,k\}\) for every integer \(k \ge 1\). The set of all cells of $G(n,m)$ is 
$$V_{n,m}=\{(i,j)\mid i\in[n],\; j\in[m] \}.$$

For every cell \( v \in V_{n,m} \), we denote by \( M(v) \) and \( VN(v) \) the \emph{Moore neighborhood} and the \emph{von Neumann neighborhood} of \( v \), respectively (see Figure~\ref{fig:neighborhood_types} for an illustration).

A cell \( u \) is said to be \emph{threatened} with contamination by two contaminated cells \( v \) and \( w \) if at least one of the following conditions holds:
\begin{itemize}
    \item[$(i)$] Both \( v \) and \( w \) belong to the von Neumann neighborhood of \( u \), i.e., \( v, w \in VN(u) \).
    \item[$(ii)$] Both \( v \) and \( w \) belong to the Moore neighborhood of \( u \), and their Moore neighborhoods intersect only at \( u \), i.e., $v, w \in M(u)$ and $M(v) \cap M(w) = \{u\}$.
\end{itemize}

Figure~\ref{fig:rules} provides an illustration of all possible contamination rules generated from Conditions $(i)$ and $(ii)$.

\begin{figure}[H]
    \centering



    \caption{Moore and von Neumann neighborhoods of an interior cell, a corner cell, and a boundary cell. 
The reference cell is shown in orange, and its neighborhood cells are shown in red.}
    \label{fig:neighborhood_types}
\end{figure}

Therefore, a cell \( u = (i,j) \in V_{n,m} \) is contaminated by two contaminated cells \( v \) and \( w \) if and only if at least one of the rules  \textup{(\texttt{a})}--\textup{(\texttt{h})} is satisfied. 
The contamination rules are presented in Figure~\ref{fig:rules} in both algebraic and illustrated form. In the illustrations, the threatened cell \(u\) is shown in orange, whereas the contaminating cells \(v\) and \(w\) are shown in red. This color convention will be used throughout the paper to distinguish threatened cells from contaminating cells.

\begin{figure}[H]
\begin{minipage}{.5\textwidth}
\begin{enumerate}
\item[\textup{(\texttt{a})}] $v=(i-1,j-1)$ and $w=(i+1,j+1)$;
\item[\textup{(\texttt{b})}] $v=(i+1,j-1)$ and $w=(i-1,j+1)$;
\item[\textup{(\texttt{c})}] $v=(i-1,j)$ and $w=(i+1,j)$;
\item[\textup{(\texttt{d})}] $v=(i,j-1)$ and $w=(i,j+1)$;
\item[\textup{(\texttt{e})}] $v=(i,j-1)$ and $w=(i-1,j)$;
\item[\textup{(\texttt{f})}] $v=(i,j-1)$ and $w=(i+1,j)$;
\item[\textup{(\texttt{g})}] $v=(i+1,j)$ and $w=(i,j+1)$;
\item[\textup{(\texttt{h})}] $v=(i-1,j)$ and $w=(i,j+1)$.
\end{enumerate}
\end{minipage}% % leave no gap
\hspace{3cm} 
\begin{minipage}{.5\textwidth}
\begin{tikzpicture}[scale=0.8pt,x=0.75pt,y=0.75pt,yscale=-1,xscale=1]
%uncomment if require: \path (0,405); %set diagram left start at 0, and has height of 405

%Shape: Square [id:dp09074002333720843] 
\draw  [fill={rgb, 255:red, 255; green, 255; blue, 255 }  ,fill opacity=1 ] (81,0) -- (101,0) -- (101,20) -- (81,20) -- cycle ;
%Shape: Square [id:dp10860874300963763] 
\draw  [fill={rgb, 255:red, 255; green, 255; blue, 255 }  ,fill opacity=1 ] (101,0) -- (121,0) -- (121,20) -- (101,20) -- cycle ;
%Shape: Square [id:dp15433623430344667] 
\draw  [fill={rgb, 255:red, 208; green, 2; blue, 27 }  ,fill opacity=1 ] (121,0) -- (141,0) -- (141,20) -- (121,20) -- cycle ;
%Shape: Square [id:dp042036081921894564] 
\draw  [fill={rgb, 255:red, 255; green, 255; blue, 255 }  ,fill opacity=1 ] (81,20) -- (101,20) -- (101,40) -- (81,40) -- cycle ;
%Shape: Square [id:dp1654565367411931] 
\draw  [fill={rgb, 255:red, 245; green, 166; blue, 35 }  ,fill opacity=1 ] (101,20) -- (121,20) -- (121,40) -- (101,40) -- cycle ;
%Shape: Square [id:dp20965684447899968] 
\draw  [fill={rgb, 255:red, 255; green, 255; blue, 255 }  ,fill opacity=1 ] (121,20) -- (141,20) -- (141,40) -- (121,40) -- cycle ;
%Shape: Square [id:dp7926020422323068] 
\draw  [fill={rgb, 255:red, 208; green, 2; blue, 27 }  ,fill opacity=1 ] (81,40) -- (101,40) -- (101,60) -- (81,60) -- cycle ;
%Shape: Square [id:dp32206209619689674] 
\draw  [fill={rgb, 255:red, 255; green, 255; blue, 255 }  ,fill opacity=1 ] (121,40) -- (141,40) -- (141,60) -- (121,60) -- cycle ;
%Shape: Square [id:dp03201539099833561] 
\draw  [fill={rgb, 255:red, 255; green, 255; blue, 255 }  ,fill opacity=1 ] (101,40) -- (121,40) -- (121,60) -- (101,60) -- cycle ;

%Shape: Square [id:dp31031372560919146] 
\draw  [fill={rgb, 255:red, 208; green, 2; blue, 27 }  ,fill opacity=1 ] (0,0) -- (20,0) -- (20,20) -- (0,20) -- cycle ;
%Shape: Square [id:dp48362084477101885] 
\draw  [fill={rgb, 255:red, 255; green, 255; blue, 255 }  ,fill opacity=1 ] (20,0) -- (40,0) -- (40,20) -- (20,20) -- cycle ;
%Shape: Square [id:dp20098591214242445] 
\draw  [fill={rgb, 255:red, 255; green, 255; blue, 255 }  ,fill opacity=1 ] (40,0) -- (60,0) -- (60,20) -- (40,20) -- cycle ;
%Shape: Square [id:dp6518931306765186] 
\draw  [fill={rgb, 255:red, 255; green, 255; blue, 255 }  ,fill opacity=1 ] (0,20) -- (20,20) -- (20,40) -- (0,40) -- cycle ;
%Shape: Square [id:dp909925399990611] 
\draw  [fill={rgb, 255:red, 245; green, 166; blue, 35 }  ,fill opacity=1 ] (20,20) -- (40,20) -- (40,40) -- (20,40) -- cycle ;
%Shape: Square [id:dp5033218631492498] 
\draw  [fill={rgb, 255:red, 255; green, 255; blue, 255 }  ,fill opacity=1 ] (40,20) -- (60,20) -- (60,40) -- (40,40) -- cycle ;
%Shape: Square [id:dp6320832716181481] 
\draw  [fill={rgb, 255:red, 255; green, 255; blue, 255 }  ,fill opacity=1 ] (0,40) -- (20,40) -- (20,60) -- (0,60) -- cycle ;
%Shape: Square [id:dp8376795013823488] 
\draw  [fill={rgb, 255:red, 208; green, 2; blue, 27 }  ,fill opacity=1 ] (40,40) -- (60,40) -- (60,60) -- (40,60) -- cycle ;
%Shape: Square [id:dp8236746366573073] 
\draw  [fill={rgb, 255:red, 255; green, 255; blue, 255 }  ,fill opacity=1 ] (20,40) -- (40,40) -- (40,60) -- (20,60) -- cycle ;

%Shape: Square [id:dp8454876782855745] 
\draw  [fill={rgb, 255:red, 255; green, 255; blue, 255 }  ,fill opacity=1 ] (80,100) -- (100,100) -- (100,120) -- (80,120) -- cycle ;
%Shape: Square [id:dp8609300826232185] 
\draw  [fill={rgb, 255:red, 208; green, 2; blue, 27 }  ,fill opacity=1 ] (100,100) -- (120,100) -- (120,120) -- (100,120) -- cycle ;
%Shape: Square [id:dp9200366658980061] 
\draw  [fill={rgb, 255:red, 255; green, 255; blue, 255 }  ,fill opacity=1 ] (120,100) -- (140,100) -- (140,120) -- (120,120) -- cycle ;
%Shape: Square [id:dp47036746867609813] 
\draw  [fill={rgb, 255:red, 208; green, 2; blue, 27 }  ,fill opacity=1 ] (80,120) -- (100,120) -- (100,140) -- (80,140) -- cycle ;
%Shape: Square [id:dp8135165291302127] 
\draw  [fill={rgb, 255:red, 245; green, 166; blue, 35 }  ,fill opacity=1 ] (100,120) -- (120,120) -- (120,140) -- (100,140) -- cycle ;
%Shape: Square [id:dp9512321183778492] 
\draw  [fill={rgb, 255:red, 255; green, 255; blue, 255 }  ,fill opacity=1 ] (120,120) -- (140,120) -- (140,140) -- (120,140) -- cycle ;
%Shape: Square [id:dp141488198179899] 
\draw  [fill={rgb, 255:red, 255; green, 255; blue, 255 }  ,fill opacity=1 ] (80,140) -- (100,140) -- (100,160) -- (80,160) -- cycle ;
%Shape: Square [id:dp4489873225318073] 
\draw  [fill={rgb, 255:red, 255; green, 255; blue, 255 }  ,fill opacity=1 ] (120,140) -- (140,140) -- (140,160) -- (120,160) -- cycle ;
%Shape: Square [id:dp26491213827310434] 
\draw  [fill={rgb, 255:red, 255; green, 255; blue, 255 }  ,fill opacity=1 ] (100,140) -- (120,140) -- (120,160) -- (100,160) -- cycle ;

%Shape: Square [id:dp3882991818855144] 
\draw  [fill={rgb, 255:red, 255; green, 255; blue, 255 }  ,fill opacity=1 ] (160,0) -- (180,0) -- (180,20) -- (160,20) -- cycle ;
%Shape: Square [id:dp14523228072355066] 
\draw  [fill={rgb, 255:red, 208; green, 2; blue, 27 }  ,fill opacity=1 ] (180,0) -- (200,0) -- (200,20) -- (180,20) -- cycle ;
%Shape: Square [id:dp33316529681919893] 
\draw  [fill={rgb, 255:red, 255; green, 255; blue, 255 }  ,fill opacity=1 ] (200,0) -- (220,0) -- (220,20) -- (200,20) -- cycle ;
%Shape: Square [id:dp3835320557858741] 
\draw  [fill={rgb, 255:red, 255; green, 255; blue, 255 }  ,fill opacity=1 ] (160,20) -- (180,20) -- (180,40) -- (160,40) -- cycle ;
%Shape: Square [id:dp4868439788782497] 
\draw  [fill={rgb, 255:red, 245; green, 166; blue, 35 }  ,fill opacity=1 ] (180,20) -- (200,20) -- (200,40) -- (180,40) -- cycle ;
%Shape: Square [id:dp29429189387951205] 
\draw  [fill={rgb, 255:red, 255; green, 255; blue, 255 }  ,fill opacity=1 ] (200,20) -- (220,20) -- (220,40) -- (200,40) -- cycle ;
%Shape: Square [id:dp13961697569644094] 
\draw  [fill={rgb, 255:red, 255; green, 255; blue, 255 }  ,fill opacity=1 ] (160,40) -- (180,40) -- (180,60) -- (160,60) -- cycle ;
%Shape: Square [id:dp5223235997210418] 
\draw  [fill={rgb, 255:red, 255; green, 255; blue, 255 }  ,fill opacity=1 ] (200,40) -- (220,40) -- (220,60) -- (200,60) -- cycle ;
%Shape: Square [id:dp5074944681845508] 
\draw  [fill={rgb, 255:red, 208; green, 2; blue, 27 }  ,fill opacity=1 ] (180,40) -- (200,40) -- (200,60) -- (180,60) -- cycle ;

%Shape: Square [id:dp11759734801330257] 
\draw  [fill={rgb, 255:red, 255; green, 255; blue, 255 }  ,fill opacity=1 ] (0,100) -- (20,100) -- (20,120) -- (0,120) -- cycle ;
%Shape: Square [id:dp13302814767939708] 
\draw  [fill={rgb, 255:red, 255; green, 255; blue, 255 }  ,fill opacity=1 ] (20,100) -- (40,100) -- (40,120) -- (20,120) -- cycle ;
%Shape: Square [id:dp830890883054282] 
\draw  [fill={rgb, 255:red, 255; green, 255; blue, 255 }  ,fill opacity=1 ] (40,100) -- (60,100) -- (60,120) -- (40,120) -- cycle ;
%Shape: Square [id:dp7248498570204149] 
\draw  [fill={rgb, 255:red, 208; green, 2; blue, 27 }  ,fill opacity=1 ] (0,120) -- (20,120) -- (20,140) -- (0,140) -- cycle ;
%Shape: Square [id:dp8105361025190854] 
\draw  [fill={rgb, 255:red, 245; green, 166; blue, 35 }  ,fill opacity=1 ] (20,120) -- (40,120) -- (40,140) -- (20,140) -- cycle ;
%Shape: Square [id:dp07330851452737153] 
\draw  [fill={rgb, 255:red, 208; green, 2; blue, 27 }  ,fill opacity=1 ] (40,120) -- (60,120) -- (60,140) -- (40,140) -- cycle ;
%Shape: Square [id:dp2136720321291854] 
\draw  [fill={rgb, 255:red, 255; green, 255; blue, 255 }  ,fill opacity=1 ] (0,140) -- (20,140) -- (20,160) -- (0,160) -- cycle ;
%Shape: Square [id:dp5511618449519238] 
\draw  [fill={rgb, 255:red, 255; green, 255; blue, 255 }  ,fill opacity=1 ] (40,140) -- (60,140) -- (60,160) -- (40,160) -- cycle ;
%Shape: Square [id:dp1391885988601036] 
\draw  [fill={rgb, 255:red, 255; green, 255; blue, 255 }  ,fill opacity=1 ] (20,140) -- (40,140) -- (40,160) -- (20,160) -- cycle ;

%Shape: Square [id:dp7229635230733416] 
\draw  [fill={rgb, 255:red, 255; green, 255; blue, 255 }  ,fill opacity=1 ] (160,100) -- (180,100) -- (180,120) -- (160,120) -- cycle ;
%Shape: Square [id:dp4079063181333258] 
\draw  [fill={rgb, 255:red, 255; green, 255; blue, 255 }  ,fill opacity=1 ] (180,100) -- (200,100) -- (200,120) -- (180,120) -- cycle ;
%Shape: Square [id:dp6597534984418663] 
\draw  [fill={rgb, 255:red, 255; green, 255; blue, 255 }  ,fill opacity=1 ] (200,100) -- (220,100) -- (220,120) -- (200,120) -- cycle ;
%Shape: Square [id:dp9763043194077632] 
\draw  [fill={rgb, 255:red, 208; green, 2; blue, 27 }  ,fill opacity=1 ] (160,120) -- (180,120) -- (180,140) -- (160,140) -- cycle ;
%Shape: Square [id:dp603034352612352] 
\draw  [fill={rgb, 255:red, 245; green, 166; blue, 35 }  ,fill opacity=1 ] (180,120) -- (200,120) -- (200,140) -- (180,140) -- cycle ;
%Shape: Square [id:dp03512470957274738] 
\draw  [fill={rgb, 255:red, 255; green, 255; blue, 255 }  ,fill opacity=1 ] (200,120) -- (220,120) -- (220,140) -- (200,140) -- cycle ;
%Shape: Square [id:dp35436258100221774] 
\draw  [fill={rgb, 255:red, 255; green, 255; blue, 255 }  ,fill opacity=1 ] (160,140) -- (180,140) -- (180,160) -- (160,160) -- cycle ;
%Shape: Square [id:dp44359626051618894] 
\draw  [fill={rgb, 255:red, 255; green, 255; blue, 255 }  ,fill opacity=1 ] (200,140) -- (220,140) -- (220,160) -- (200,160) -- cycle ;
%Shape: Square [id:dp27009039284296343] 
\draw  [fill={rgb, 255:red, 208; green, 2; blue, 27 }  ,fill opacity=1 ] (180,140) -- (200,140) -- (200,160) -- (180,160) -- cycle ;

%Shape: Square [id:dp436201404316918] 
\draw  [fill={rgb, 255:red, 255; green, 255; blue, 255 }  ,fill opacity=1 ] (40,200) -- (60,200) -- (60,220) -- (40,220) -- cycle ;
%Shape: Square [id:dp18071852422183543] 
\draw  [fill={rgb, 255:red, 255; green, 255; blue, 255 }  ,fill opacity=1 ] (60,200) -- (80,200) -- (80,220) -- (60,220) -- cycle ;
%Shape: Square [id:dp037160668714995904] 
\draw  [fill={rgb, 255:red, 255; green, 255; blue, 255 }  ,fill opacity=1 ] (80,200) -- (100,200) -- (100,220) -- (80,220) -- cycle ;
%Shape: Square [id:dp08421292263327462] 
\draw  [fill={rgb, 255:red, 255; green, 255; blue, 255 }  ,fill opacity=1 ] (40,220) -- (60,220) -- (60,240) -- (40,240) -- cycle ;
%Shape: Square [id:dp993284529997245] 
\draw  [fill={rgb, 255:red, 245; green, 166; blue, 35 }  ,fill opacity=1 ] (60,220) -- (80,220) -- (80,240) -- (60,240) -- cycle ;
%Shape: Square [id:dp328360610838905] 
\draw  [fill={rgb, 255:red, 208; green, 2; blue, 27 }  ,fill opacity=1 ] (80,220) -- (100,220) -- (100,240) -- (80,240) -- cycle ;
%Shape: Square [id:dp20280284883145305] 
\draw  [fill={rgb, 255:red, 255; green, 255; blue, 255 }  ,fill opacity=1 ] (40,240) -- (60,240) -- (60,260) -- (40,260) -- cycle ;
%Shape: Square [id:dp9249547175940422] 
\draw  [fill={rgb, 255:red, 255; green, 255; blue, 255 }  ,fill opacity=1 ] (80,240) -- (100,240) -- (100,260) -- (80,260) -- cycle ;
%Shape: Square [id:dp8465284144427834] 
\draw  [fill={rgb, 255:red, 208; green, 2; blue, 27 }  ,fill opacity=1 ] (60,240) -- (80,240) -- (80,260) -- (60,260) -- cycle ;

%Shape: Square [id:dp18302863687890647] 
\draw  [fill={rgb, 255:red, 255; green, 255; blue, 255 }  ,fill opacity=1 ] (120,200) -- (140,200) -- (140,220) -- (120,220) -- cycle ;
%Shape: Square [id:dp1739422673742066] 
\draw  [fill={rgb, 255:red, 208; green, 2; blue, 27 }  ,fill opacity=1 ] (140,200) -- (160,200) -- (160,220) -- (140,220) -- cycle ;
%Shape: Square [id:dp22802462801768764] 
\draw  [fill={rgb, 255:red, 255; green, 255; blue, 255 }  ,fill opacity=1 ] (160,200) -- (180,200) -- (180,220) -- (160,220) -- cycle ;
%Shape: Square [id:dp9227052903421089] 
\draw  [fill={rgb, 255:red, 255; green, 255; blue, 255 }  ,fill opacity=1 ] (120,220) -- (140,220) -- (140,240) -- (120,240) -- cycle ;
%Shape: Square [id:dp9300943875923264] 
\draw  [fill={rgb, 255:red, 245; green, 166; blue, 35 }  ,fill opacity=1 ] (140,220) -- (160,220) -- (160,240) -- (140,240) -- cycle ;
%Shape: Square [id:dp6049240347648464] 
\draw  [fill={rgb, 255:red, 208; green, 2; blue, 27 }  ,fill opacity=1 ] (160,220) -- (180,220) -- (180,240) -- (160,240) -- cycle ;
%Shape: Square [id:dp06072050114333605] 
\draw  [fill={rgb, 255:red, 255; green, 255; blue, 255 }  ,fill opacity=1 ] (120,240) -- (140,240) -- (140,260) -- (120,260) -- cycle ;
%Shape: Square [id:dp9900876768534042] 
\draw  [fill={rgb, 255:red, 255; green, 255; blue, 255 }  ,fill opacity=1 ] (160,240) -- (180,240) -- (180,260) -- (160,260) -- cycle ;
%Shape: Square [id:dp6790861334410312] 
\draw  [fill={rgb, 255:red, 255; green, 255; blue, 255 }  ,fill opacity=1 ] (140,240) -- (160,240) -- (160,260) -- (140,260) -- cycle ;

% Text Node
\draw (22,63.4) node [anchor=north west][inner sep=0.75pt]    {\footnotesize\textup{(\texttt{a})}};
% Text Node
\draw (103,63.4) node [anchor=north west][inner sep=0.75pt]    {\footnotesize\textup{(\texttt{b})}};
% Text Node
\draw (182,63.4) node [anchor=north west][inner sep=0.75pt]    {\footnotesize\textup{(\texttt{c})}};
% Text Node
\draw (22,163.4) node [anchor=north west][inner sep=0.75pt]    {\footnotesize\textup{(\texttt{d})}};
% Text Node
\draw (182,163.4) node [anchor=north west][inner sep=0.75pt]    {\footnotesize\textup{(\texttt{f})}};
% Text Node
\draw (102,163.4) node [anchor=north west][inner sep=0.75pt]    {\footnotesize\textup{(\texttt{e})}};
% Text Node
\draw (142,264.4) node [anchor=north west][inner sep=0.75pt]    {\footnotesize\textup{(\texttt{h})}};
% Text Node
\draw (62,263.4) node [anchor=north west][inner sep=0.75pt]    {\footnotesize\textup{(\texttt{g})}};

\end{tikzpicture}\end{minipage}
    \caption{The contamination rules.}
    \label{fig:rules}
\end{figure}

The power contamination process is described in Algorithm~\ref{alg:power_conta_process}. 
The input is an initial contaminating set of cells \( S \), and the output is the final contaminated set of cells, denoted by \( C(S) \).

As can be observed, if $n=1$, only rule~\textup{(\texttt{d})} may apply. If $n=2$, rules~\textup{(\texttt{a})}--\textup{(\texttt{c})} do not apply.
A corner cell can be contaminated by exactly one rule, whereas a non-corner boundary cell
can be contaminated only by rules~\textup{(\texttt{c})}--\textup{(\texttt{h})}. 

\begin{algorithm}[H]
\caption{Power Contamination Process.}
\label{alg:power_conta_process}
\begin{algorithmic}[1]
\Require An initial contaminating set of cells \( S \subseteq V_{n,m} \)
\Ensure The final contaminated set of cells \( C(S)\subseteq V_{n,m} \)
\State initialize \( C(S) \gets S \);
\While{there exists a cell \( u \in V_{n,m} \setminus C(S) \) such that condition~$(i)$ or~$(ii)$ is satisfied with respect to \( C(S) \)}
    \State \( C(S) \gets C(S) \cup \{u\} \);
\EndWhile
\State \Return \( C(S) \).
\end{algorithmic}
\end{algorithm}

A subset of cells $S\subseteq V_{n,m}$ is called a contaminating set (or a feasible solution to the power contamination problem) if a full contamination of $G(n,m)$ can be achieved from $S$ under the contamination rules~\textup{(\texttt{a})}--\textup{(\texttt{h})}. The \emph{power contamination problem} (PCP) is to find the power contamination number $\gamma_{c}(G(n,m))$, that is, the minimum cardinality of a contaminating set, denoted by
\[
\gamma_{c}(G(n,m)) = \min \left\{  |S| \;\middle|\; S \subseteq V_{n,m} \ \text{is a contaminating set of } G(n,m)  \right\}.
\]
If $|S|=\gamma_{c}(G(n,m))$, then $S$ is called an optimal solution to the PCP.

This propagation process can be viewed as a bootstrap-percolation-type process \cite{Morris2017}, a well-studied class of monotone infection models in statistical physics and probability. In classical bootstrap percolation, a vertex becomes infected once it has at least \(r\) infected neighbors. In contrast, the power contamination model uses more restrictive local geometric rules: a cell becomes contaminated only when two contaminated cells appear in one of the specific relative positions described above.

The main objective in \cite{Ainouche2021} is to determine the minimum size of an initial contamination set \(S\) whose propagation contaminates the whole grid. As part of that work, the authors proposed the following conjecture for the contamination number \(\gamma_c(G(n,m))\), aiming to characterize its exact value for all positive integers \(n\) and \(m\).

\begin{conjecture}[Ainouche and Bouroubi \cite{Ainouche2021}] \label{conj}
Let \( G(n,m) \) be the \( n \times m \) grid graph. For all integers \( n, m \geq 1 \), we have 
\begin{equation}
        \gamma_c(G(n,m)) =
    \begin{cases}
        \max\left\{ \left\lfloor \dfrac{m}{2} \right\rfloor, \left\lfloor \dfrac{n}{2} \right\rfloor \right\} + 1, & \text{if $m$ and $n$ have the same parity}, \\[1.0em]
        \max\left\{ \left\lceil \dfrac{m}{2} \right\rceil, \left\lceil \dfrac{n}{2} \right\rceil \right\} + 1, & \text{otherwise}.
    \end{cases}
\end{equation}
\end{conjecture}

\subsection{Terminology and special configurations}

We now introduce additional terminology concerning rows, columns, boundary cells, rectangular blocks, and zig-zag configurations. These notions will allow us to describe both obstruction phenomena and constructive contamination patterns in a concise way.

Let $X\subseteq V_{n,m}$ be a set of cells. We say that $X$ \emph{fully contaminates} the grid $G(n,m)$ if $C(X)=V_{n,m}$. Equivalently, the contamination process starting from $X$ ends with every cell of the grid contaminated.

A row $i\in[n]$ is called \emph{nonempty with respect to $X$} if there exists $j\in[m]$ such that
$(i,j)\in X$; otherwise it is called \emph{empty}. Similarly, a column $j\in[m]$ is called
\emph{nonempty with respect to $X$} if there exists $i\in[n]$ such that $(i,j)\in X$; otherwise
it is called \emph{empty}. When $X=S$ is an initial contamination set, we also say that an empty
column is \emph{clean}.

The \emph{boundary rows} of $G(n,m)$ are the first and last rows, namely rows $1$ and $n$.
The \emph{boundary columns} are the first and last columns, namely columns $1$ and $m$.
A cell $(i,j)\in V_{n,m}$ is called a \emph{boundary cell} if it belongs to a boundary row or a
boundary column, that is, if $i\in\{1,n\}$ or $j\in\{1,m\}$.
A boundary cell is called a \emph{corner cell} if $(i,j)\in\{(1,1),(1,m),(n,1),(n,m)\}$. A cell that is not a boundary cell is called an \emph{interior cell}.

For integers $1\le a\le b\le n$ and $1\le c\le d\le m$, we denote by
\[
R[a,b;c,d]:=\{(i,j)\in V_{n,m}: a\le i\le b,\ c\le j\le d\}
\]
the corresponding \emph{rectangular block} of $G(n,m)$.
A set of cells $X\subseteq V_{n,m}$ is said to form a \emph{rectangle} if
$X=R[a,b;c,d]$ for some integers $a,b,c,d$ as above. If needed for emphasis, such a set may also be called a
\emph{perfect rectangle}.

Assume now that $n\ge 2$. The \emph{zig-zag path} of $G(n,m)$ is the set
\[
Z_{n,m}:=\{(\rho_j,j): 1\le j\le m\},
\]
where the sequence $(\rho_j)_{j=1}^m$ is defined recursively as follows: $\rho_1=1$, and for each $j\ge 1$, the value $\rho_{j+1}$ is obtained from $\rho_j$ by moving one row down
or one row up, reversing direction whenever row $1$ or row $n$ is reached. Equivalently,
the sequence of row indices follows the pattern $1,2,\dots,n,n-1,\dots,2,1,2,\dots$ truncated after $m$ terms. 

An \emph{alternating zig-zag set} is any subset of $Z_{n,m}$ obtained by selecting every other cell
along the zig-zag path. In particular, the set obtained by selecting the first, third, fifth, and so on
cells of $Z_{n,m}$ will be called the \emph{odd alternating zig-zag set}. If one wishes to keep the
terminology used later in the paper, this may also be called a \emph{discrete contaminated zig-zag path}.

Figure~\ref{fig:zig-zig-path} illustrates a zig-zag path in $G(4,9)$ and its associated alternating zig-zag path. In each column, the zig-zag path contains exactly one cell, and the alternating zig-zag path is obtained by selecting every other cell along it.

\begin{figure}[H]
    \centering

\begin{tikzpicture}[scale=0.8,x=0.75pt,y=0.75pt,yscale=-1,xscale=1]
%uncomment if require: \path (0,127); %set diagram left start at 0, and has height of 127

%Shape: Square [id:dp7248017694229862] 
\draw  [fill={rgb, 255:red, 208; green, 2; blue, 27 }  ,fill opacity=1 ] (0,0) -- (20,0) -- (20,20) -- (0,20) -- cycle ;
%Shape: Square [id:dp0062196380363455095] 
\draw  [fill={rgb, 255:red, 255; green, 255; blue, 255 }  ,fill opacity=1 ] (20,20) -- (40,20) -- (40,40) -- (20,40) -- cycle ;
%Shape: Square [id:dp3760915810068066] 
\draw  [fill={rgb, 255:red, 208; green, 2; blue, 27 }  ,fill opacity=1 ] (20,20) -- (40,20) -- (40,40) -- (20,40) -- cycle ;
%Shape: Square [id:dp5025650120664742] 
\draw  [fill={rgb, 255:red, 208; green, 2; blue, 27 }  ,fill opacity=1 ] (40,40) -- (60,40) -- (60,60) -- (40,60) -- cycle ;
%Shape: Square [id:dp6564486366167208] 
\draw  [fill={rgb, 255:red, 208; green, 2; blue, 27 }  ,fill opacity=1 ] (60,60) -- (80,60) -- (80,80) -- (60,80) -- cycle ;
%Shape: Square [id:dp5150220000604298] 
\draw  [fill={rgb, 255:red, 255; green, 255; blue, 255 }  ,fill opacity=1 ] (20,0) -- (40,0) -- (40,20) -- (20,20) -- cycle ;
%Shape: Square [id:dp3244200315536643] 
\draw  [fill={rgb, 255:red, 255; green, 255; blue, 255 }  ,fill opacity=1 ] (40,20) -- (60,20) -- (60,40) -- (40,40) -- cycle ;
%Shape: Square [id:dp3149163128318324] 
\draw  [fill={rgb, 255:red, 255; green, 255; blue, 255 }  ,fill opacity=1 ] (60,40) -- (80,40) -- (80,60) -- (60,60) -- cycle ;
%Shape: Square [id:dp1327138107351573] 
\draw  [fill={rgb, 255:red, 255; green, 255; blue, 255 }  ,fill opacity=1 ] (80,60) -- (100,60) -- (100,80) -- (80,80) -- cycle ;
%Shape: Square [id:dp6686477485556301] 
\draw  [fill={rgb, 255:red, 255; green, 255; blue, 255 }  ,fill opacity=1 ] (40,0) -- (60,0) -- (60,20) -- (40,20) -- cycle ;
%Shape: Square [id:dp19626852540220163] 
\draw  [fill={rgb, 255:red, 255; green, 255; blue, 255 }  ,fill opacity=1 ] (60,20) -- (80,20) -- (80,40) -- (60,40) -- cycle ;
%Shape: Square [id:dp10815912607862721] 
\draw  [fill={rgb, 255:red, 208; green, 2; blue, 27 }  ,fill opacity=1 ] (80,40) -- (100,40) -- (100,60) -- (80,60) -- cycle ;
%Shape: Square [id:dp502499897580889] 
\draw  [fill={rgb, 255:red, 255; green, 255; blue, 255 }  ,fill opacity=1 ] (60,0) -- (80,0) -- (80,20) -- (60,20) -- cycle ;
%Shape: Square [id:dp7281607545364162] 
\draw  [fill={rgb, 255:red, 255; green, 255; blue, 255 }  ,fill opacity=1 ] (80,20) -- (100,20) -- (100,40) -- (80,40) -- cycle ;
%Shape: Square [id:dp6842417222815231] 
\draw  [fill={rgb, 255:red, 255; green, 255; blue, 255 }  ,fill opacity=1 ] (80,0) -- (100,0) -- (100,20) -- (80,20) -- cycle ;
%Shape: Square [id:dp5154458986189694] 
\draw  [fill={rgb, 255:red, 255; green, 255; blue, 255 }  ,fill opacity=1 ] (0,20) -- (20,20) -- (20,40) -- (0,40) -- cycle ;
%Shape: Square [id:dp8029972171918554] 
\draw  [fill={rgb, 255:red, 255; green, 255; blue, 255 }  ,fill opacity=1 ] (20,40) -- (40,40) -- (40,60) -- (20,60) -- cycle ;
%Shape: Square [id:dp25375903331411065] 
\draw  [fill={rgb, 255:red, 255; green, 255; blue, 255 }  ,fill opacity=1 ] (40,60) -- (60,60) -- (60,80) -- (40,80) -- cycle ;
%Shape: Square [id:dp17281611326785806] 
\draw  [fill={rgb, 255:red, 255; green, 255; blue, 255 }  ,fill opacity=1 ] (0,40) -- (20,40) -- (20,60) -- (0,60) -- cycle ;
%Shape: Square [id:dp9106096652512166] 
\draw  [fill={rgb, 255:red, 255; green, 255; blue, 255 }  ,fill opacity=1 ] (0,60) -- (20,60) -- (20,80) -- (0,80) -- cycle ;
%Shape: Square [id:dp6652646583655739] 
\draw  [fill={rgb, 255:red, 255; green, 255; blue, 255 }  ,fill opacity=1 ] (20,60) -- (40,60) -- (40,80) -- (20,80) -- cycle ;
%Shape: Square [id:dp18575029674328314] 
\draw  [fill={rgb, 255:red, 208; green, 2; blue, 27 }  ,fill opacity=1 ] (100,20) -- (120,20) -- (120,40) -- (100,40) -- cycle ;
%Shape: Square [id:dp2905763162433135] 
\draw  [fill={rgb, 255:red, 255; green, 255; blue, 255 }  ,fill opacity=1 ] (120,40) -- (140,40) -- (140,60) -- (120,60) -- cycle ;
%Shape: Square [id:dp47963310738748244] 
\draw  [fill={rgb, 255:red, 255; green, 255; blue, 255 }  ,fill opacity=1 ] (140,60) -- (160,60) -- (160,80) -- (140,80) -- cycle ;
%Shape: Square [id:dp1531529536570826] 
\draw  [fill={rgb, 255:red, 255; green, 255; blue, 255 }  ,fill opacity=1 ] (160,60) -- (180,60) -- (180,80) -- (160,80) -- cycle ;
%Shape: Square [id:dp3360741194136305] 
\draw  [fill={rgb, 255:red, 255; green, 255; blue, 255 }  ,fill opacity=1 ] (100,0) -- (120,0) -- (120,20) -- (100,20) -- cycle ;
%Shape: Square [id:dp5195389931273364] 
\draw  [fill={rgb, 255:red, 255; green, 255; blue, 255 }  ,fill opacity=1 ] (120,20) -- (140,20) -- (140,40) -- (120,40) -- cycle ;
%Shape: Square [id:dp42969426374552544] 
\draw  [fill={rgb, 255:red, 255; green, 255; blue, 255 }  ,fill opacity=1 ] (140,40) -- (160,40) -- (160,60) -- (140,60) -- cycle ;
%Shape: Square [id:dp44111068356895244] 
\draw  [fill={rgb, 255:red, 208; green, 2; blue, 27 }  ,fill opacity=1 ] (120,0) -- (140,0) -- (140,20) -- (120,20) -- cycle ;
%Shape: Square [id:dp884838817830462] 
\draw  [fill={rgb, 255:red, 208; green, 2; blue, 27 }  ,fill opacity=1 ] (140,20) -- (160,20) -- (160,40) -- (140,40) -- cycle ;
%Shape: Square [id:dp7956768692699359] 
\draw  [fill={rgb, 255:red, 208; green, 2; blue, 27 }  ,fill opacity=1 ] (160,40) -- (180,40) -- (180,60) -- (160,60) -- cycle ;
%Shape: Square [id:dp8106269718566221] 
\draw  [fill={rgb, 255:red, 255; green, 255; blue, 255 }  ,fill opacity=1 ] (140,0) -- (160,0) -- (160,20) -- (140,20) -- cycle ;
%Shape: Square [id:dp1186097983326233] 
\draw  [fill={rgb, 255:red, 255; green, 255; blue, 255 }  ,fill opacity=1 ] (160,20) -- (180,20) -- (180,40) -- (160,40) -- cycle ;
%Shape: Square [id:dp8745484418360052] 
\draw  [fill={rgb, 255:red, 255; green, 255; blue, 255 }  ,fill opacity=1 ] (160,0) -- (180,0) -- (180,20) -- (160,20) -- cycle ;
%Shape: Square [id:dp058234907379793865] 
\draw  [fill={rgb, 255:red, 255; green, 255; blue, 255 }  ,fill opacity=1 ] (100,40) -- (120,40) -- (120,60) -- (100,60) -- cycle ;
%Shape: Square [id:dp2867656463695467] 
\draw  [fill={rgb, 255:red, 255; green, 255; blue, 255 }  ,fill opacity=1 ] (120,60) -- (140,60) -- (140,80) -- (120,80) -- cycle ;
%Shape: Square [id:dp29656732660617546] 
\draw  [fill={rgb, 255:red, 255; green, 255; blue, 255 }  ,fill opacity=1 ] (100,60) -- (120,60) -- (120,80) -- (100,80) -- cycle ;

%Shape: Square [id:dp34874106491317947] 
\draw  [fill={rgb, 255:red, 208; green, 2; blue, 27 }  ,fill opacity=1 ] (250,0) -- (270,0) -- (270,20) -- (250,20) -- cycle ;
%Shape: Square [id:dp1738007143548297] 
\draw  [fill={rgb, 255:red, 255; green, 255; blue, 255 }  ,fill opacity=1 ] (270,20) -- (290,20) -- (290,40) -- (270,40) -- cycle ;
%Shape: Square [id:dp9993324926714957] 
\draw  [fill={rgb, 255:red, 255; green, 255; blue, 255 }  ,fill opacity=1 ] (270,20) -- (290,20) -- (290,40) -- (270,40) -- cycle ;
%Shape: Square [id:dp8033905169015503] 
\draw  [fill={rgb, 255:red, 208; green, 2; blue, 27 }  ,fill opacity=1 ] (290,40) -- (310,40) -- (310,60) -- (290,60) -- cycle ;
%Shape: Square [id:dp21105540006989654] 
\draw  [fill={rgb, 255:red, 255; green, 255; blue, 255 }  ,fill opacity=1 ] (310,60) -- (330,60) -- (330,80) -- (310,80) -- cycle ;
%Shape: Square [id:dp1636956717324145] 
\draw  [fill={rgb, 255:red, 255; green, 255; blue, 255 }  ,fill opacity=1 ] (270,0) -- (290,0) -- (290,20) -- (270,20) -- cycle ;
%Shape: Square [id:dp4861941377082917] 
\draw  [fill={rgb, 255:red, 255; green, 255; blue, 255 }  ,fill opacity=1 ] (290,20) -- (310,20) -- (310,40) -- (290,40) -- cycle ;
%Shape: Square [id:dp855718614421167] 
\draw  [fill={rgb, 255:red, 255; green, 255; blue, 255 }  ,fill opacity=1 ] (310,40) -- (330,40) -- (330,60) -- (310,60) -- cycle ;
%Shape: Square [id:dp641578556666826] 
\draw  [fill={rgb, 255:red, 255; green, 255; blue, 255 }  ,fill opacity=1 ] (330,60) -- (350,60) -- (350,80) -- (330,80) -- cycle ;
%Shape: Square [id:dp3479087530844107] 
\draw  [fill={rgb, 255:red, 255; green, 255; blue, 255 }  ,fill opacity=1 ] (290,0) -- (310,0) -- (310,20) -- (290,20) -- cycle ;
%Shape: Square [id:dp08734082219140116] 
\draw  [fill={rgb, 255:red, 255; green, 255; blue, 255 }  ,fill opacity=1 ] (310,20) -- (330,20) -- (330,40) -- (310,40) -- cycle ;
%Shape: Square [id:dp05325666205572133] 
\draw  [fill={rgb, 255:red, 208; green, 2; blue, 27 }  ,fill opacity=1 ] (330,40) -- (350,40) -- (350,60) -- (330,60) -- cycle ;
%Shape: Square [id:dp6359511284180275] 
\draw  [fill={rgb, 255:red, 255; green, 255; blue, 255 }  ,fill opacity=1 ] (310,0) -- (330,0) -- (330,20) -- (310,20) -- cycle ;
%Shape: Square [id:dp12510532858365364] 
\draw  [fill={rgb, 255:red, 255; green, 255; blue, 255 }  ,fill opacity=1 ] (330,20) -- (350,20) -- (350,40) -- (330,40) -- cycle ;
%Shape: Square [id:dp007282318122335596] 
\draw  [fill={rgb, 255:red, 255; green, 255; blue, 255 }  ,fill opacity=1 ] (330,0) -- (350,0) -- (350,20) -- (330,20) -- cycle ;
%Shape: Square [id:dp3509911857567545] 
\draw  [fill={rgb, 255:red, 255; green, 255; blue, 255 }  ,fill opacity=1 ] (250,20) -- (270,20) -- (270,40) -- (250,40) -- cycle ;
%Shape: Square [id:dp19831304753329837] 
\draw  [fill={rgb, 255:red, 255; green, 255; blue, 255 }  ,fill opacity=1 ] (270,40) -- (290,40) -- (290,60) -- (270,60) -- cycle ;
%Shape: Square [id:dp670472950081618] 
\draw  [fill={rgb, 255:red, 255; green, 255; blue, 255 }  ,fill opacity=1 ] (290,60) -- (310,60) -- (310,80) -- (290,80) -- cycle ;
%Shape: Square [id:dp8861504177224704] 
\draw  [fill={rgb, 255:red, 255; green, 255; blue, 255 }  ,fill opacity=1 ] (250,40) -- (270,40) -- (270,60) -- (250,60) -- cycle ;
%Shape: Square [id:dp03329339715687807] 
\draw  [fill={rgb, 255:red, 255; green, 255; blue, 255 }  ,fill opacity=1 ] (250,60) -- (270,60) -- (270,80) -- (250,80) -- cycle ;
%Shape: Square [id:dp9175044301642667] 
\draw  [fill={rgb, 255:red, 255; green, 255; blue, 255 }  ,fill opacity=1 ] (270,60) -- (290,60) -- (290,80) -- (270,80) -- cycle ;
%Shape: Square [id:dp3307690279494784] 
\draw  [fill={rgb, 255:red, 255; green, 255; blue, 255 }  ,fill opacity=1 ] (350,20) -- (370,20) -- (370,40) -- (350,40) -- cycle ;
%Shape: Square [id:dp7834543543230885] 
\draw  [fill={rgb, 255:red, 255; green, 255; blue, 255 }  ,fill opacity=1 ] (370,40) -- (390,40) -- (390,60) -- (370,60) -- cycle ;
%Shape: Square [id:dp595304967996694] 
\draw  [fill={rgb, 255:red, 255; green, 255; blue, 255 }  ,fill opacity=1 ] (390,60) -- (410,60) -- (410,80) -- (390,80) -- cycle ;
%Shape: Square [id:dp24550335769339693] 
\draw  [fill={rgb, 255:red, 255; green, 255; blue, 255 }  ,fill opacity=1 ] (410,60) -- (430,60) -- (430,80) -- (410,80) -- cycle ;
%Shape: Square [id:dp8034030697729642] 
\draw  [fill={rgb, 255:red, 255; green, 255; blue, 255 }  ,fill opacity=1 ] (350,0) -- (370,0) -- (370,20) -- (350,20) -- cycle ;
%Shape: Square [id:dp5336018185640046] 
\draw  [fill={rgb, 255:red, 255; green, 255; blue, 255 }  ,fill opacity=1 ] (370,20) -- (390,20) -- (390,40) -- (370,40) -- cycle ;
%Shape: Square [id:dp451878573104302] 
\draw  [fill={rgb, 255:red, 255; green, 255; blue, 255 }  ,fill opacity=1 ] (390,40) -- (410,40) -- (410,60) -- (390,60) -- cycle ;
%Shape: Square [id:dp2003787497379259] 
\draw  [fill={rgb, 255:red, 208; green, 2; blue, 27 }  ,fill opacity=1 ] (370,0) -- (390,0) -- (390,20) -- (370,20) -- cycle ;
%Shape: Square [id:dp8763864788791746] 
\draw  [fill={rgb, 255:red, 255; green, 255; blue, 255 }  ,fill opacity=1 ] (390,20) -- (410,20) -- (410,40) -- (390,40) -- cycle ;
%Shape: Square [id:dp7865725878116572] 
\draw  [fill={rgb, 255:red, 208; green, 2; blue, 27 }  ,fill opacity=1 ] (410,40) -- (430,40) -- (430,60) -- (410,60) -- cycle ;
%Shape: Square [id:dp7050625819117682] 
\draw  [fill={rgb, 255:red, 255; green, 255; blue, 255 }  ,fill opacity=1 ] (390,0) -- (410,0) -- (410,20) -- (390,20) -- cycle ;
%Shape: Square [id:dp5530266690497023] 
\draw  [fill={rgb, 255:red, 255; green, 255; blue, 255 }  ,fill opacity=1 ] (410,20) -- (430,20) -- (430,40) -- (410,40) -- cycle ;
%Shape: Square [id:dp8308634671242872] 
\draw  [fill={rgb, 255:red, 255; green, 255; blue, 255 }  ,fill opacity=1 ] (410,0) -- (430,0) -- (430,20) -- (410,20) -- cycle ;
%Shape: Square [id:dp4433967064288622] 
\draw  [fill={rgb, 255:red, 255; green, 255; blue, 255 }  ,fill opacity=1 ] (350,40) -- (370,40) -- (370,60) -- (350,60) -- cycle ;
%Shape: Square [id:dp4233534081462593] 
\draw  [fill={rgb, 255:red, 255; green, 255; blue, 255 }  ,fill opacity=1 ] (370,60) -- (390,60) -- (390,80) -- (370,80) -- cycle ;
%Shape: Square [id:dp6837112544547799] 
\draw  [fill={rgb, 255:red, 255; green, 255; blue, 255 }  ,fill opacity=1 ] (350,60) -- (370,60) -- (370,80) -- (350,80) -- cycle ;

% Text Node
\draw (55.5,93) node [anchor=north west][inner sep=0.75pt]  [font=\footnotesize,] [align=left] {zig-zag path};
% Text Node
\draw (255,89) node [anchor=north west][inner sep=0.75pt]  [font=\footnotesize,] [align=left] {alternating zig-zag path};

\end{tikzpicture}
\caption{A zig-zag path in $G(4,9)$ (left) and the corresponding alternating zig-zag path (right).}

    \label{fig:zig-zig-path}
\end{figure}

Let $S\subseteq V_{n,m}$ be a set of cells. We say that $S$ is \emph{supported on the odd columns}
if every cell of $S$ lies in a column of odd index. Equivalently, $S$ contains no cell in any even column.

For a word $w=w_1w_2\cdots w_t$ over a finite alphabet, a \emph{factor} of $w$ is a contiguous block $w_iw_{i+1}\cdots w_j$, with $1\le i\le j\le t$.

\section{Structural Properties of Power Contamination on Grids}
\label{sec:structural}

Before determining the exact value of the power contamination number, we establish several structural properties of the contamination process on grids. Some of these results describe configurations that prevent complete contamination, while others identify local or global patterns that force propagation. Together, they provide the main tools used in the proof of the exact formula.

\subsection{Obstructions and boundary constraints}
We first focus on configurations that obstruct propagation. In particular, we show that certain empty boundary or gap patterns prevent full contamination, and we record basic geometric properties of contaminated and uncontaminated rectangles.

\begin{observation}\label{obse:observations}
The following local configurations force complete contamination:\begin{itemize}

    \item[\textup{(A)}]  Two contaminated cells positioned according to rule~\textup{(\texttt{a})} or~\textup{(\texttt{b})} contaminate \(G(3,3)\).\label{item:remarkonG33}
    \item[\textup{(B)}]  Two contaminated cells positioned according to rule~\textup{(\texttt{c})} (respectively, rule~\textup{(\texttt{d})})   contaminate  \(G(3,1)\) (respectively, \(G(1,3)\)).
    \item[\textup{(C)}]  Two contaminated cells positioned according to any of the rules~\textup{(\texttt{e})}--\textup{(\texttt{h})}   contaminate  \(G(2,2)\).
\end{itemize}
\end{observation}

\begin{lemma}\label{lem:no-two-empty}
If there exists at least one consecutive pair of clean columns (resp., rows), then the power contamination process cannot result in a fully contaminated grid.
\end{lemma}

\begin{proof}
Suppose there is a pair of consecutive clean columns (resp., rows). By the contamination rules~\textup{(\texttt{a})}--\textup{(\texttt{h})}, contaminating a new cell requires exactly two contaminated cells in specific relative positions. In this case, none of the rules can be applied to contaminate any cell within these two consecutive columns (resp., rows), even if both are adjacent to fully contaminated columns (resp., rows). Therefore, at least these two columns (resp., rows) will remain uncontaminated at the end of the process.\qedhere
\end{proof}

Figure~\ref{fig:two-empty-columns} illustrates the obstruction described in Lemma~\ref{lem:no-two-empty}, where a   pair of consecutive empty columns blocks the contamination process: no rule can create a contaminated cell inside this gap, and therefore full contamination is impossible.

\begin{figure}[H]
    \centering

\begin{tikzpicture}[scale=0.8pt,x=0.75pt,y=0.75pt,yscale=-1,xscale=1]
%uncomment if require: \path (0,166); %set diagram left start at 0, and has height of 166

%Shape: Rectangle [id:dp16188334559811612] 
\draw  [fill={rgb, 255:red, 208; green, 2; blue, 27 }  ,fill opacity=1 ] (52,23) -- (72.03,23) -- (72.03,43.02) -- (52,43.02) -- cycle ;
%Shape: Rectangle [id:dp47774625002391047] 
\draw  [fill={rgb, 255:red, 208; green, 2; blue, 27 }  ,fill opacity=1 ] (72.03,23) -- (92.05,23) -- (92.05,43.02) -- (72.03,43.02) -- cycle ;
%Shape: Rectangle [id:dp3554052160020851] 
\draw  [fill={rgb, 255:red, 208; green, 2; blue, 27 }  ,fill opacity=1 ] (92.05,23) -- (112.08,23) -- (112.08,43.02) -- (92.05,43.02) -- cycle ;
%Shape: Rectangle [id:dp6400594131205064] 
\draw  [fill={rgb, 255:red, 255; green, 255; blue, 255 }  ,fill opacity=1 ] (112.08,23) -- (132.1,23) -- (132.1,43.02) -- (112.08,43.02) -- cycle ;
%Shape: Rectangle [id:dp9686666796256964] 
\draw  [fill={rgb, 255:red, 255; green, 255; blue, 255 }  ,fill opacity=1 ] (132.1,23) -- (152.13,23) -- (152.13,43.02) -- (132.1,43.02) -- cycle ;
%Shape: Rectangle [id:dp8112889060592892] 
\draw  [fill={rgb, 255:red, 208; green, 2; blue, 27 }  ,fill opacity=1 ] (152.13,23) -- (172.16,23) -- (172.16,43.02) -- (152.13,43.02) -- cycle ;
%Shape: Rectangle [id:dp8814590992069109] 
\draw  [fill={rgb, 255:red, 208; green, 2; blue, 27 }  ,fill opacity=1 ] (52,43.02) -- (72.03,43.02) -- (72.03,63.04) -- (52,63.04) -- cycle ;
%Shape: Rectangle [id:dp5173251230600107] 
\draw  [fill={rgb, 255:red, 208; green, 2; blue, 27 }  ,fill opacity=1 ] (72.03,43.02) -- (92.05,43.02) -- (92.05,63.04) -- (72.03,63.04) -- cycle ;
%Shape: Rectangle [id:dp6906029977895936] 
\draw  [fill={rgb, 255:red, 208; green, 2; blue, 27 }  ,fill opacity=1 ] (92.05,43.02) -- (112.08,43.02) -- (112.08,63.04) -- (92.05,63.04) -- cycle ;
%Shape: Rectangle [id:dp3817176023077604] 
\draw  [fill={rgb, 255:red, 255; green, 255; blue, 255 }  ,fill opacity=1 ] (112.08,43.02) -- (132.1,43.02) -- (132.1,63.04) -- (112.08,63.04) -- cycle ;
%Shape: Rectangle [id:dp002395250805335536] 
\draw  [fill={rgb, 255:red, 255; green, 255; blue, 255 }  ,fill opacity=1 ] (132.1,43.02) -- (152.13,43.02) -- (152.13,63.04) -- (132.1,63.04) -- cycle ;
%Shape: Rectangle [id:dp33523836218019154] 
\draw  [fill={rgb, 255:red, 208; green, 2; blue, 27 }  ,fill opacity=1 ] (152.13,43.02) -- (172.16,43.02) -- (172.16,63.04) -- (152.13,63.04) -- cycle ;
%Shape: Rectangle [id:dp9411322649702227] 
\draw  [fill={rgb, 255:red, 208; green, 2; blue, 27 }  ,fill opacity=1 ] (52,63.04) -- (72.03,63.04) -- (72.03,83.06) -- (52,83.06) -- cycle ;
%Shape: Rectangle [id:dp14371448846587853] 
\draw  [fill={rgb, 255:red, 208; green, 2; blue, 27 }  ,fill opacity=1 ] (72.03,63.04) -- (92.05,63.04) -- (92.05,83.06) -- (72.03,83.06) -- cycle ;
%Shape: Rectangle [id:dp2798993082313035] 
\draw  [fill={rgb, 255:red, 208; green, 2; blue, 27 }  ,fill opacity=1 ] (92.05,63.04) -- (112.08,63.04) -- (112.08,83.06) -- (92.05,83.06) -- cycle ;
%Shape: Rectangle [id:dp5952502466086744] 
\draw  [fill={rgb, 255:red, 255; green, 255; blue, 255 }  ,fill opacity=1 ] (112.08,63.04) -- (132.1,63.04) -- (132.1,83.06) -- (112.08,83.06) -- cycle ;
%Shape: Rectangle [id:dp4950987852563338] 
\draw  [fill={rgb, 255:red, 255; green, 255; blue, 255 }  ,fill opacity=1 ] (132.1,63.04) -- (152.13,63.04) -- (152.13,83.06) -- (132.1,83.06) -- cycle ;
%Shape: Rectangle [id:dp028616771828666288] 
\draw  [fill={rgb, 255:red, 208; green, 2; blue, 27 }  ,fill opacity=1 ] (152.13,63.04) -- (172.16,63.04) -- (172.16,83.06) -- (152.13,83.06) -- cycle ;
%Shape: Rectangle [id:dp40683556431754875] 
\draw  [fill={rgb, 255:red, 208; green, 2; blue, 27 }  ,fill opacity=1 ] (52,83.06) -- (72.03,83.06) -- (72.03,103.08) -- (52,103.08) -- cycle ;
%Shape: Rectangle [id:dp40403890987023994] 
\draw  [fill={rgb, 255:red, 208; green, 2; blue, 27 }  ,fill opacity=1 ] (72.03,83.06) -- (92.05,83.06) -- (92.05,103.08) -- (72.03,103.08) -- cycle ;
%Shape: Rectangle [id:dp5760635864891746] 
\draw  [fill={rgb, 255:red, 208; green, 2; blue, 27 }  ,fill opacity=1 ] (92.05,83.06) -- (112.08,83.06) -- (112.08,103.08) -- (92.05,103.08) -- cycle ;
%Shape: Rectangle [id:dp6173688667411725] 
\draw  [fill={rgb, 255:red, 255; green, 255; blue, 255 }  ,fill opacity=1 ] (112.08,83.06) -- (132.1,83.06) -- (132.1,103.08) -- (112.08,103.08) -- cycle ;
%Shape: Rectangle [id:dp3784311327979115] 
\draw  [fill={rgb, 255:red, 255; green, 255; blue, 255 }  ,fill opacity=1 ] (132.1,83.06) -- (152.13,83.06) -- (152.13,103.08) -- (132.1,103.08) -- cycle ;
%Shape: Rectangle [id:dp9370457013310105] 
\draw  [fill={rgb, 255:red, 208; green, 2; blue, 27 }  ,fill opacity=1 ] (152.13,83.06) -- (172.16,83.06) -- (172.16,103.08) -- (152.13,103.08) -- cycle ;
%Shape: Rectangle [id:dp5457278512011732] 
\draw  [fill={rgb, 255:red, 208; green, 2; blue, 27 }  ,fill opacity=1 ] (212.21,83.06) -- (232.24,83.06) -- (232.24,103.08) -- (212.21,103.08) -- cycle ;
%Shape: Rectangle [id:dp05974965315018421] 
\draw  [fill={rgb, 255:red, 208; green, 2; blue, 27 }  ,fill opacity=1 ] (212.21,63.04) -- (232.24,63.04) -- (232.24,83.06) -- (212.21,83.06) -- cycle ;
%Shape: Rectangle [id:dp3852296948338776] 
\draw  [fill={rgb, 255:red, 208; green, 2; blue, 27 }  ,fill opacity=1 ] (212.21,43.02) -- (232.24,43.02) -- (232.24,63.04) -- (212.21,63.04) -- cycle ;
%Shape: Rectangle [id:dp8545157569825825] 
\draw  [fill={rgb, 255:red, 208; green, 2; blue, 27 }  ,fill opacity=1 ] (212.21,23) -- (232.24,23) -- (232.24,43.02) -- (212.21,43.02) -- cycle ;
%Shape: Rectangle [id:dp8416755717038482] 
\draw  [fill={rgb, 255:red, 208; green, 2; blue, 27 }  ,fill opacity=1 ] (172.16,83.06) -- (192.18,83.06) -- (192.18,103.08) -- (172.16,103.08) -- cycle ;
%Shape: Rectangle [id:dp19685271376045765] 
\draw  [fill={rgb, 255:red, 208; green, 2; blue, 27 }  ,fill opacity=1 ] (172.16,63.04) -- (192.18,63.04) -- (192.18,83.06) -- (172.16,83.06) -- cycle ;
%Shape: Rectangle [id:dp21684429916710357] 
\draw  [fill={rgb, 255:red, 208; green, 2; blue, 27 }  ,fill opacity=1 ] (172.16,43.02) -- (192.18,43.02) -- (192.18,63.04) -- (172.16,63.04) -- cycle ;
%Shape: Rectangle [id:dp9361061948644571] 
\draw  [fill={rgb, 255:red, 208; green, 2; blue, 27 }  ,fill opacity=1 ] (172.16,23) -- (192.18,23) -- (192.18,43.02) -- (172.16,43.02) -- cycle ;
%Shape: Rectangle [id:dp022038086091199283] 
\draw  [fill={rgb, 255:red, 208; green, 2; blue, 27 }  ,fill opacity=1 ] (192.18,23) -- (212.21,23) -- (212.21,43.02) -- (192.18,43.02) -- cycle ;
%Shape: Rectangle [id:dp6058253576480122] 
\draw  [fill={rgb, 255:red, 208; green, 2; blue, 27 }  ,fill opacity=1 ] (192.18,43.02) -- (212.21,43.02) -- (212.21,63.04) -- (192.18,63.04) -- cycle ;
%Shape: Rectangle [id:dp6113914261650146] 
\draw  [fill={rgb, 255:red, 208; green, 2; blue, 27 }  ,fill opacity=1 ] (192.18,63.04) -- (212.21,63.04) -- (212.21,83.06) -- (192.18,83.06) -- cycle ;
%Shape: Rectangle [id:dp18491401186836742] 
\draw  [fill={rgb, 255:red, 208; green, 2; blue, 27 }  ,fill opacity=1 ] (192.18,83.06) -- (212.21,83.06) -- (212.21,103.08) -- (192.18,103.08) -- cycle ;

%Shape: Rectangle [id:dp27939659864553246] 
\draw  [fill={rgb, 255:red, 208; green, 2; blue, 27 }  ,fill opacity=1 ] (323,23) -- (343.03,23) -- (343.03,43.02) -- (323,43.02) -- cycle ;
%Shape: Rectangle [id:dp4179445328437257] 
\draw  [fill={rgb, 255:red, 208; green, 2; blue, 27 }  ,fill opacity=1 ] (343.03,23) -- (363.05,23) -- (363.05,43.02) -- (343.03,43.02) -- cycle ;
%Shape: Rectangle [id:dp7838989374384921] 
\draw  [fill={rgb, 255:red, 208; green, 2; blue, 27 }  ,fill opacity=1 ] (363.05,23) -- (383.08,23) -- (383.08,43.02) -- (363.05,43.02) -- cycle ;
%Shape: Rectangle [id:dp5532902702540186] 
\draw  [fill={rgb, 255:red, 208; green, 2; blue, 27 }  ,fill opacity=1 ] (383.08,23) -- (403.1,23) -- (403.1,43.02) -- (383.08,43.02) -- cycle ;
%Shape: Rectangle [id:dp22756510939576757] 
\draw  [fill={rgb, 255:red, 208; green, 2; blue, 27 }  ,fill opacity=1 ] (403.1,23) -- (423.13,23) -- (423.13,43.02) -- (403.1,43.02) -- cycle ;
%Shape: Rectangle [id:dp5315140783583212] 
\draw  [fill={rgb, 255:red, 208; green, 2; blue, 27 }  ,fill opacity=1 ] (423.13,23) -- (443.16,23) -- (443.16,43.02) -- (423.13,43.02) -- cycle ;
%Shape: Rectangle [id:dp5094150529360342] 
\draw  [fill={rgb, 255:red, 255; green, 255; blue, 255 }  ,fill opacity=1 ] (323,43.02) -- (343.03,43.02) -- (343.03,63.04) -- (323,63.04) -- cycle ;
%Shape: Rectangle [id:dp9292302018044067] 
\draw  [fill={rgb, 255:red, 255; green, 255; blue, 255 }  ,fill opacity=1 ] (343.03,43.02) -- (363.05,43.02) -- (363.05,63.04) -- (343.03,63.04) -- cycle ;
%Shape: Rectangle [id:dp6113594710063501] 
\draw  [fill={rgb, 255:red, 255; green, 255; blue, 255 }  ,fill opacity=1 ] (363.05,43.02) -- (383.08,43.02) -- (383.08,63.04) -- (363.05,63.04) -- cycle ;
%Shape: Rectangle [id:dp555590510068436] 
\draw  [fill={rgb, 255:red, 255; green, 255; blue, 255 }  ,fill opacity=1 ] (383.08,43.02) -- (403.1,43.02) -- (403.1,63.04) -- (383.08,63.04) -- cycle ;
%Shape: Rectangle [id:dp9483176323080373] 
\draw  [fill={rgb, 255:red, 255; green, 255; blue, 255 }  ,fill opacity=1 ] (403.1,43.02) -- (423.13,43.02) -- (423.13,63.04) -- (403.1,63.04) -- cycle ;
%Shape: Rectangle [id:dp04762869365827238] 
\draw  [fill={rgb, 255:red, 255; green, 255; blue, 255 }  ,fill opacity=1 ] (423.13,43.02) -- (443.16,43.02) -- (443.16,63.04) -- (423.13,63.04) -- cycle ;
%Shape: Rectangle [id:dp5730027464553906] 
\draw  [fill={rgb, 255:red, 255; green, 255; blue, 255 }  ,fill opacity=1 ] (323,63.04) -- (343.03,63.04) -- (343.03,83.06) -- (323,83.06) -- cycle ;
%Shape: Rectangle [id:dp4484954980762019] 
\draw  [fill={rgb, 255:red, 255; green, 255; blue, 255 }  ,fill opacity=1 ] (343.03,63.04) -- (363.05,63.04) -- (363.05,83.06) -- (343.03,83.06) -- cycle ;
%Shape: Rectangle [id:dp44022984095909146] 
\draw  [fill={rgb, 255:red, 255; green, 255; blue, 255 }  ,fill opacity=1 ] (363.05,63.04) -- (383.08,63.04) -- (383.08,83.06) -- (363.05,83.06) -- cycle ;
%Shape: Rectangle [id:dp3616816247619383] 
\draw  [fill={rgb, 255:red, 255; green, 255; blue, 255 }  ,fill opacity=1 ] (383.08,63.04) -- (403.1,63.04) -- (403.1,83.06) -- (383.08,83.06) -- cycle ;
%Shape: Rectangle [id:dp746857197177426] 
\draw  [fill={rgb, 255:red, 255; green, 255; blue, 255 }  ,fill opacity=1 ] (403.1,63.04) -- (423.13,63.04) -- (423.13,83.06) -- (403.1,83.06) -- cycle ;
%Shape: Rectangle [id:dp1480548738005989] 
\draw  [fill={rgb, 255:red, 255; green, 255; blue, 255 }  ,fill opacity=1 ] (423.13,63.04) -- (443.16,63.04) -- (443.16,83.06) -- (423.13,83.06) -- cycle ;
%Shape: Rectangle [id:dp8514497592918102] 
\draw  [fill={rgb, 255:red, 208; green, 2; blue, 27 }  ,fill opacity=1 ] (323,83.06) -- (343.03,83.06) -- (343.03,103.08) -- (323,103.08) -- cycle ;
%Shape: Rectangle [id:dp5539811187477053] 
\draw  [fill={rgb, 255:red, 208; green, 2; blue, 27 }  ,fill opacity=1 ] (343.03,83.06) -- (363.05,83.06) -- (363.05,103.08) -- (343.03,103.08) -- cycle ;
%Shape: Rectangle [id:dp921683281591734] 
\draw  [fill={rgb, 255:red, 208; green, 2; blue, 27 }  ,fill opacity=1 ] (363.05,83.06) -- (383.08,83.06) -- (383.08,103.08) -- (363.05,103.08) -- cycle ;
%Shape: Rectangle [id:dp7912597268243069] 
\draw  [fill={rgb, 255:red, 208; green, 2; blue, 27 }  ,fill opacity=1 ] (383.08,83.06) -- (403.1,83.06) -- (403.1,103.08) -- (383.08,103.08) -- cycle ;
%Shape: Rectangle [id:dp6327466371134254] 
\draw  [fill={rgb, 255:red, 208; green, 2; blue, 27 }  ,fill opacity=1 ] (403.1,83.06) -- (423.13,83.06) -- (423.13,103.08) -- (403.1,103.08) -- cycle ;
%Shape: Rectangle [id:dp12344902238373212] 
\draw  [fill={rgb, 255:red, 208; green, 2; blue, 27 }  ,fill opacity=1 ] (423.13,83.06) -- (443.16,83.06) -- (443.16,103.08) -- (423.13,103.08) -- cycle ;
%Shape: Rectangle [id:dp7786259023517641] 
\draw  [fill={rgb, 255:red, 208; green, 2; blue, 27 }  ,fill opacity=1 ] (483.21,83.06) -- (503.24,83.06) -- (503.24,103.08) -- (483.21,103.08) -- cycle ;
%Shape: Rectangle [id:dp7234335401967693] 
\draw  [fill={rgb, 255:red, 255; green, 255; blue, 255 }  ,fill opacity=1 ] (483.21,63.04) -- (503.24,63.04) -- (503.24,83.06) -- (483.21,83.06) -- cycle ;
%Shape: Rectangle [id:dp13621429946256258] 
\draw  [fill={rgb, 255:red, 255; green, 255; blue, 255 }  ,fill opacity=1 ] (483.21,43.02) -- (503.24,43.02) -- (503.24,63.04) -- (483.21,63.04) -- cycle ;
%Shape: Rectangle [id:dp36233589387534004] 
\draw  [fill={rgb, 255:red, 208; green, 2; blue, 27 }  ,fill opacity=1 ] (483.21,23) -- (503.24,23) -- (503.24,43.02) -- (483.21,43.02) -- cycle ;
%Shape: Rectangle [id:dp8195219858625182] 
\draw  [fill={rgb, 255:red, 208; green, 2; blue, 27 }  ,fill opacity=1 ] (443.16,83.06) -- (463.18,83.06) -- (463.18,103.08) -- (443.16,103.08) -- cycle ;
%Shape: Rectangle [id:dp2594424259817598] 
\draw  [fill={rgb, 255:red, 255; green, 255; blue, 255 }  ,fill opacity=1 ] (443.16,63.04) -- (463.18,63.04) -- (463.18,83.06) -- (443.16,83.06) -- cycle ;
%Shape: Rectangle [id:dp09806317590847913] 
\draw  [fill={rgb, 255:red, 255; green, 255; blue, 255 }  ,fill opacity=1 ] (443.16,43.02) -- (463.18,43.02) -- (463.18,63.04) -- (443.16,63.04) -- cycle ;
%Shape: Rectangle [id:dp24450497604215093] 
\draw  [fill={rgb, 255:red, 208; green, 2; blue, 27 }  ,fill opacity=1 ] (443.16,23) -- (463.18,23) -- (463.18,43.02) -- (443.16,43.02) -- cycle ;
%Shape: Rectangle [id:dp14006161891720192] 
\draw  [fill={rgb, 255:red, 208; green, 2; blue, 27 }  ,fill opacity=1 ] (463.18,23) -- (483.21,23) -- (483.21,43.02) -- (463.18,43.02) -- cycle ;
%Shape: Rectangle [id:dp1511569302506508] 
\draw  [fill={rgb, 255:red, 255; green, 255; blue, 255 }  ,fill opacity=1 ] (463.18,43.02) -- (483.21,43.02) -- (483.21,63.04) -- (463.18,63.04) -- cycle ;
%Shape: Rectangle [id:dp675752122695894] 
\draw  [fill={rgb, 255:red, 255; green, 255; blue, 255 }  ,fill opacity=1 ] (463.18,63.04) -- (483.21,63.04) -- (483.21,83.06) -- (463.18,83.06) -- cycle ;
%Shape: Rectangle [id:dp08237226277233178] 
\draw  [fill={rgb, 255:red, 208; green, 2; blue, 27 }  ,fill opacity=1 ] (463.18,83.06) -- (483.21,83.06) -- (483.21,103.08) -- (463.18,103.08) -- cycle ;

\end{tikzpicture}

    \caption{A pair of consecutive empty columns blocks the contamination process.}
    \label{fig:two-empty-columns}
\end{figure}

\begin{proposition}\label{prop:rectangular-boundary}
Let $X$ be the set of contaminated cells at some stage of the contamination process on $G(n,m)$.

\begin{enumerate}
    \item[\textup{(A)}] If $X$ forms a rectangle, then no further cell can be contaminated.

    \item[\textup{(B)}] If the uncontaminated set $V_{n,m}\setminus X$ forms a rectangle and two adjacent outer
    sides of this rectangle are completely contaminated, then the process continues until the whole
    rectangle is contaminated.

    \item[\textup{(C)}] If some boundary row or some boundary column contains no cell of the initial contamination
    set $S$, then $S$ does not fully contaminate the grid.
\end{enumerate}
\end{proposition}

\begin{proof}
  Suppose that $X=R[a,b;c,d]$ is a rectangular block, and let $u\notin X$.
    If $u$ could be contaminated, then two contaminated cells of $X$ would have to satisfy one of
    the rules~(a)--(h) with respect to $u$.

    However, in each of the rules~(a)--(h), the target cell lies geometrically between, or at the
    corner determined by, the two contaminating cells. Since $X$ is a rectangle, every such target
    cell necessarily also belongs to $X$. Therefore no cell outside $X$ can be contaminated, and the
    process stops. This proves \textup{(A)}.

     Let $U:=V_{n,m}\setminus X$, and suppose that $U=R[a,b;c,d]$ is a rectangular block.
    By symmetry, we may assume that the cells immediately to the right of $U$ and immediately below
    $U$ are contaminated, that is, the outer sides $\{(i,d+1): a\le i\le b\}$ and $\{(b+1,j): c\le j\le d\}$ are fully contaminated whenever these indices are defined.

    We prove that all cells of $U$ become contaminated by induction on $\delta(i,j):=(b-i)+(d-j)$.

    For $\delta(i,j)=0$, we have $(i,j)=(b,d)$. Since $(b,d+1)$ and $(b+1,d)$ are contaminated,
    rule~(g) contaminates $(b,d)$.

    Now let $(i,j)\in U$ with $\delta(i,j)\ge 1$, and assume inductively that every cell
    $(i',j')\in U$ with $\delta(i',j')<\delta(i,j)$ is already contaminated.
    Then either $(i+1,j)$ is contaminated (if $i<b$), or it lies on the contaminated outer side
    below $U$; similarly, either $(i,j+1)$ is contaminated (if $j<d$), or it lies on the contaminated
    outer side to the right of $U$. Hence rule~(g) contaminates $(i,j)$. This completes the proof of \textup{(B)}. 

    Therefore every cell of $U$ becomes contaminated, and so the whole grid becomes contaminated.

     By symmetry, it suffices to consider the case where the top row contains no cell of $S$.
    We show that no cell of the top row can ever become contaminated.

    Indeed, let $(1,j)$ be a cell in the top row. Since there is no row above it, rules~(a), (b),
    (c), and~(h) cannot apply to $(1,j)$. The remaining applicable rules are~(d), (f), and~(g),
    and each of them requires at least one already contaminated cell in the top row, namely
    $(1,j-1)$ or $(1,j+1)$.
    Thus a first contaminated cell in the top row cannot appear.

    Since the top row is initially empty and cannot acquire a first contaminated cell, it remains
    uncontaminated throughout the process. Consequently, the grid cannot become fully contaminated. This completes the proof of \textup{(C)}.\qedhere
\end{proof}

Figure~\ref{fig:rectangular-boundary} illustrates the three parts of Proposition~\ref{prop:rectangular-boundary}: (A) a contaminated rectangle that cannot expand; (B) a rectangular hole with two adjacent contaminated outer sides that is eventually filled; (C) an empty boundary row that prevents full contamination

\begin{figure}[H]
    \centering

\begin{tikzpicture}[scale=0.8,x=0.75pt,y=0.75pt,yscale=-1,xscale=1]
%uncomment if require: \path (0,115); %set diagram left start at 0, and has height of 115

%Shape: Square [id:dp9817101244473294] 
\draw  [fill={rgb, 255:red, 208; green, 2; blue, 27 }  ,fill opacity=1 ] (0,0) -- (20,0) -- (20,20) -- (0,20) -- cycle ;
%Shape: Square [id:dp9797785282764424] 
\draw  [fill={rgb, 255:red, 255; green, 255; blue, 255 }  ,fill opacity=1 ] (20,20) -- (40,20) -- (40,40) -- (20,40) -- cycle ;
%Shape: Square [id:dp6522246187008667] 
\draw  [fill={rgb, 255:red, 208; green, 2; blue, 27 }  ,fill opacity=1 ] (20,20) -- (40,20) -- (40,40) -- (20,40) -- cycle ;
%Shape: Square [id:dp11817482226331277] 
\draw  [fill={rgb, 255:red, 208; green, 2; blue, 27 }  ,fill opacity=1 ] (20,0) -- (40,0) -- (40,20) -- (20,20) -- cycle ;
%Shape: Square [id:dp2777621744384249] 
\draw  [fill={rgb, 255:red, 255; green, 255; blue, 255 }  ,fill opacity=1 ] (40,40) -- (60,40) -- (60,60) -- (40,60) -- cycle ;
%Shape: Square [id:dp23826901877465212] 
\draw  [fill={rgb, 255:red, 255; green, 255; blue, 255 }  ,fill opacity=1 ] (40,20) -- (60,20) -- (60,40) -- (40,40) -- cycle ;
%Shape: Square [id:dp63485739817116] 
\draw  [fill={rgb, 255:red, 255; green, 255; blue, 255 }  ,fill opacity=1 ] (60,40) -- (80,40) -- (80,60) -- (60,60) -- cycle ;
%Shape: Square [id:dp10331362674285582] 
\draw  [fill={rgb, 255:red, 255; green, 255; blue, 255 }  ,fill opacity=1 ] (40,0) -- (60,0) -- (60,20) -- (40,20) -- cycle ;
%Shape: Square [id:dp14245886940025776] 
\draw  [fill={rgb, 255:red, 255; green, 255; blue, 255 }  ,fill opacity=1 ] (60,20) -- (80,20) -- (80,40) -- (60,40) -- cycle ;
%Shape: Square [id:dp7004541816895946] 
\draw  [fill={rgb, 255:red, 255; green, 255; blue, 255 }  ,fill opacity=1 ] (60,0) -- (80,0) -- (80,20) -- (60,20) -- cycle ;
%Shape: Square [id:dp630533005231086] 
\draw  [fill={rgb, 255:red, 208; green, 2; blue, 27 }  ,fill opacity=1 ] (0,20) -- (20,20) -- (20,40) -- (0,40) -- cycle ;
%Shape: Square [id:dp46132533444876267] 
\draw  [fill={rgb, 255:red, 255; green, 255; blue, 255 }  ,fill opacity=1 ] (20,40) -- (40,40) -- (40,60) -- (20,60) -- cycle ;
%Shape: Square [id:dp331004561767557] 
\draw  [fill={rgb, 255:red, 255; green, 255; blue, 255 }  ,fill opacity=1 ] (0,40) -- (20,40) -- (20,60) -- (0,60) -- cycle ;

%Shape: Square [id:dp8480098354670347] 
\draw  [fill={rgb, 255:red, 255; green, 255; blue, 255 }  ,fill opacity=1 ] (150,0) -- (170,0) -- (170,20) -- (150,20) -- cycle ;
%Shape: Square [id:dp6708627962158638] 
\draw  [fill={rgb, 255:red, 255; green, 255; blue, 255 }  ,fill opacity=1 ] (170,20) -- (190,20) -- (190,40) -- (170,40) -- cycle ;
%Shape: Square [id:dp6971090738061523] 
\draw  [fill={rgb, 255:red, 255; green, 255; blue, 255 }  ,fill opacity=1 ] (170,20) -- (190,20) -- (190,40) -- (170,40) -- cycle ;
%Shape: Square [id:dp40041989054942717] 
\draw  [fill={rgb, 255:red, 255; green, 255; blue, 255 }  ,fill opacity=1 ] (170,0) -- (190,0) -- (190,20) -- (170,20) -- cycle ;
%Shape: Square [id:dp8317444324948802] 
\draw  [fill={rgb, 255:red, 208; green, 2; blue, 27 }  ,fill opacity=1 ] (190,40) -- (210,40) -- (210,60) -- (190,60) -- cycle ;
%Shape: Square [id:dp9015144503885133] 
\draw  [fill={rgb, 255:red, 208; green, 2; blue, 27 }  ,fill opacity=1 ] (190,20) -- (210,20) -- (210,40) -- (190,40) -- cycle ;
%Shape: Square [id:dp9221491656307721] 
\draw  [fill={rgb, 255:red, 208; green, 2; blue, 27 }  ,fill opacity=1 ] (210,40) -- (230,40) -- (230,60) -- (210,60) -- cycle ;
%Shape: Square [id:dp6445068931901907] 
\draw  [fill={rgb, 255:red, 208; green, 2; blue, 27 }  ,fill opacity=1 ] (190,0) -- (210,0) -- (210,20) -- (190,20) -- cycle ;
%Shape: Square [id:dp0015395020534492154] 
\draw  [fill={rgb, 255:red, 208; green, 2; blue, 27 }  ,fill opacity=1 ] (210,20) -- (230,20) -- (230,40) -- (210,40) -- cycle ;
%Shape: Square [id:dp39631547687921476] 
\draw  [fill={rgb, 255:red, 208; green, 2; blue, 27 }  ,fill opacity=1 ] (210,0) -- (230,0) -- (230,20) -- (210,20) -- cycle ;
%Shape: Square [id:dp620055315862897] 
\draw  [fill={rgb, 255:red, 255; green, 255; blue, 255 }  ,fill opacity=1 ] (150,20) -- (170,20) -- (170,40) -- (150,40) -- cycle ;
%Shape: Square [id:dp009459254165517406] 
\draw  [fill={rgb, 255:red, 208; green, 2; blue, 27 }  ,fill opacity=1 ] (170,40) -- (190,40) -- (190,60) -- (170,60) -- cycle ;
%Shape: Square [id:dp5861502936416405] 
\draw  [fill={rgb, 255:red, 208; green, 2; blue, 27 }  ,fill opacity=1 ] (150,40) -- (170,40) -- (170,60) -- (150,60) -- cycle ;

%Shape: Square [id:dp9225817439077744] 
\draw  [fill={rgb, 255:red, 208; green, 2; blue, 27 }  ,fill opacity=1 ] (300,0) -- (320,0) -- (320,20) -- (300,20) -- cycle ;
%Shape: Square [id:dp04282434839328664] 
\draw  [fill={rgb, 255:red, 208; green, 2; blue, 27 }  ,fill opacity=1 ] (320,20) -- (340,20) -- (340,40) -- (320,40) -- cycle ;
%Shape: Square [id:dp9813336273073587] 
\draw  [fill={rgb, 255:red, 208; green, 2; blue, 27 }  ,fill opacity=1 ] (320,20) -- (340,20) -- (340,40) -- (320,40) -- cycle ;
%Shape: Square [id:dp7172402470956827] 
\draw  [fill={rgb, 255:red, 208; green, 2; blue, 27 }  ,fill opacity=1 ] (320,0) -- (340,0) -- (340,20) -- (320,20) -- cycle ;
%Shape: Square [id:dp8015862114576164] 
\draw  [fill={rgb, 255:red, 255; green, 255; blue, 255 }  ,fill opacity=1 ] (340,40) -- (360,40) -- (360,60) -- (340,60) -- cycle ;
%Shape: Square [id:dp38116873579148236] 
\draw  [fill={rgb, 255:red, 208; green, 2; blue, 27 }  ,fill opacity=1 ] (340,20) -- (360,20) -- (360,40) -- (340,40) -- cycle ;
%Shape: Square [id:dp3181451548901215] 
\draw  [fill={rgb, 255:red, 255; green, 255; blue, 255 }  ,fill opacity=1 ] (360,40) -- (380,40) -- (380,60) -- (360,60) -- cycle ;
%Shape: Square [id:dp9871892688638879] 
\draw  [fill={rgb, 255:red, 208; green, 2; blue, 27 }  ,fill opacity=1 ] (340,0) -- (360,0) -- (360,20) -- (340,20) -- cycle ;
%Shape: Square [id:dp28352068614369896] 
\draw  [fill={rgb, 255:red, 208; green, 2; blue, 27 }  ,fill opacity=1 ] (360,20) -- (380,20) -- (380,40) -- (360,40) -- cycle ;
%Shape: Square [id:dp032610981008099005] 
\draw  [fill={rgb, 255:red, 208; green, 2; blue, 27 }  ,fill opacity=1 ] (360,0) -- (380,0) -- (380,20) -- (360,20) -- cycle ;
%Shape: Square [id:dp7790072595832658] 
\draw  [fill={rgb, 255:red, 208; green, 2; blue, 27 }  ,fill opacity=1 ] (300,20) -- (320,20) -- (320,40) -- (300,40) -- cycle ;
%Shape: Square [id:dp505131996418692] 
\draw  [fill={rgb, 255:red, 255; green, 255; blue, 255 }  ,fill opacity=1 ] (320,40) -- (340,40) -- (340,60) -- (320,60) -- cycle ;
%Shape: Square [id:dp7241220471287115] 
\draw  [fill={rgb, 255:red, 255; green, 255; blue, 255 }  ,fill opacity=1 ] (300,40) -- (320,40) -- (320,60) -- (300,60) -- cycle ;

% Text Node
\draw (30.5,72) node [anchor=north west][inner sep=0.75pt]  [font=\small,] [align=left] {(A)};
% Text Node
\draw (180.5,72) node [anchor=north west][inner sep=0.75pt]  [font=\small,] [align=left] {(B)};
% Text Node
\draw (330,72) node [anchor=north west][inner sep=0.75pt]  [font=\small,] [align=left] {(C)};

\end{tikzpicture}
    \caption{Illustration of three parts of Proposition~\ref{prop:rectangular-boundary}.}
    \label{fig:rectangular-boundary}
\end{figure}

From Proposition~\ref{prop:rectangular-boundary}(C), we obtain the following result.
\begin{corollary} \label{coro:boundary}
    Each boundary of the grid $G(n,m)$ must contain at least one contaminated cell in $S$ in order for the grid to become fully contaminated.  
    If there exists a  boundary without any contaminated cell in $S$, the grid will never be fully contaminated.
\end{corollary}

\subsection{Propagating configurations}
We next turn to constructive propagation mechanisms. We show that several specific initial configurations, such as diagonal and zig-zag patterns, force the contamination process to spread across the whole grid. These results will later yield explicit upper-bound constructions.

\begin{lemma}\label{lem:diagonal-square}
Let \( G(m,m) \) be a square grid. If the initial contaminating set of cells lies on the main diagonal of the grid, i.e.,
\[
S = \{(i,i) \mid 1 \leq i \leq m\},
\]
then the contamination process results in the entire grid becoming contaminated.
\end{lemma}

\begin{proof}
Starting from the main diagonal, rules \textup{(\texttt{a})}, \textup{(\texttt{b})}, \textup{(\texttt{f})}, and \textup{(\texttt{h})} can be applied to contaminate, at each step, the two rows immediately above and below the already contaminated cells. Repeating this process propagates the contamination outward in both directions until all rows are contaminated. Since each contaminated row subsequently contaminates its entire set of columns by the same rules, the whole grid becomes contaminated. \qedhere 
\end{proof}
Figure~\ref{fig:main-diagonal} illustrates the propagation described in Lemma~\ref{lem:diagonal-square}.
\begin{figure}[H]
    \centering

\tikzset{every picture/.style={line width=0.75pt}} %set default line width to 0.75pt        


\caption{A main-diagonal initial contamination set in the square grid $G(m,m)$ forces full contamination of the grid.}\label{fig:main-diagonal}  
\end{figure}

\begin{lemma}\label{lem:full-column-propagation}
Let $1 \le j < m$, and suppose that column $j$ is fully contaminated. If column $j+1$
contains at least one contaminated cell, then column $j+1$ becomes fully contaminated.
\end{lemma}

\begin{proof}
Let $(r,j+1)$ be a contaminated cell in column $j+1$.

We first show that contamination propagates upward inside column $j+1$.
If $r>1$ and $(r,j+1)$ is contaminated, then $(r-1,j)$ is contaminated because column $j$
is fully contaminated. Hence rule~(f) contaminates $(r-1,j+1)$.
Iterating this argument yields that all cells $(1,j+1),(2,j+1),\dots,(r-1,j+1)$ become contaminated.

Similarly, contamination propagates downward inside column $j+1$.
If $r<n$ and $(r,j+1)$ is contaminated, then $(r+1,j)$ is contaminated, and rule~(e)
contaminates $(r+1,j+1)$. Iterating gives that all cells $(r+1,j+1),(r+2,j+1),\dots,(n,j+1)$ become contaminated.

Therefore every cell of column $j+1$ becomes contaminated, so column $j+1$ is fully contaminated.\qedhere
\end{proof}
\begin{lemma}\label{lem:zigzag-full}
Let $G(n,m)$ be a rectangular grid. Consider an initial contaminating set of cells $S$
forming a zig-zag path that starts from the top-left cell, descends to the bottom row of
the grid at an angle of $\pi/4$, then ascends to the top row at the same angle, and repeats
this pattern across the grid. Such an initial configuration leads to the entire grid becoming
contaminated.
\end{lemma}

\begin{proof}
Let $S=\{(r_j,j)\mid 1\le j\le m\}$,  where $(r_j,j)$ denotes the unique cell of the zig-zag path in column $j$. Since the path starts at the top-left corner and initially descends by one row at each step,
its restriction to the first $n$ columns is exactly the main diagonal $(1,1),(2,2),\dots,(n,n)$  of the square subgrid induced by columns $1,\dots,n$.
Hence, by Lemma~\ref{lem:diagonal-square}, this $n\times n$ square becomes fully contaminated.

We now prove by induction on $j$ that every column $j\in\{n,n+1,\dots,m\}$ becomes fully contaminated.
The base case $j=n$ has already been established.

Assume that column $j$ is fully contaminated for some $j\ge n$ with $j<m$.
By construction of the zig-zag path, column $j+1$ contains the contaminated cell $(r_{j+1},j+1)\in S$.
Therefore, by Lemma~\ref{lem:full-column-propagation}, column $j+1$ becomes fully contaminated. 

By induction, every column from $n$ to $m$ becomes fully contaminated.
Since the first $n$ columns were already fully contaminated, the whole grid $G(n,m)$ becomes fully contaminated. \qedhere
\end{proof}

Figure~\ref{fig:zigzag-propagation} illustrates the propagation described in Lemma~\ref{lem:zigzag-full} on the grid $G(4,9)$.

\begin{figure}[H]
    \centering


       \caption{A zig-zag initial contamination set in the rectangular grid $G(4,9)$ forces full contamination of the grid.}
\label{fig:zigzag-propagation}

\end{figure}

\section{Exact Value of the Power Contamination Number}\label{sec:value}

We now determine the exact value of the power contamination number for every grid \(G(n,m)\) with \(m \ge n \ge 1\). We begin by disproving Conjecture~\ref{conj} through a counterexample, then establish matching upper and lower bounds, and finally derive the exact formula together with some immediate consequences.

\subsection{A counterexample to the former conjecture}

    We  remark that the second part of the conjecture, namely, the case where $n$ and $m$ have different parities does not hold. The following counterexample suffices to disprove it. Consider the grid $G(4,5)$. According to the conjecture, since $n=4$ and $m=5$ have different parities and $m>n$, we should have
\[
\gamma_c(G(4,5)) = \left\lceil \frac{5}{2} \right\rceil + 1 = 3 + 1 = 4.
\]
However, this is not the case, because $G(4,5)$ can be fully contaminated using only three initially contaminated cells, as illustrated in Figure~\ref{fig:counter_example}.
\begin{figure}[H]
    \centering



    \caption{A counterexample to Conjecture~1: the grid $G(4,5)$ can be fully contaminated from an initial set of only three cells.}

    \label{fig:counter_example}
\end{figure}

 \subsection{Upper and lower bounds}
We now derive two complementary estimates for \(\gamma_c(G(n,m))\). The upper bound is obtained from explicit constructive contamination patterns, while the lower bound follows from structural restrictions on the placement of contaminated cells.

\begin{proposition}\label{prop:path-contamination-number}
For every integer $m\ge 1$, the power contamination number of the path-grid $G(1,m)$ satisfies
\begin{equation}
    \gamma_c(G(1,m))=
\begin{cases}
1, & \text{if } m=1,\\[2mm]
\gamma_c(G(1,m-4))+2, & \text{if } m\ge 4,
\end{cases}
\end{equation}
with initial values $\gamma_c(G(1,1))=1$, and  $\gamma_c(G(1,2))=\gamma_c(G(1,3))=2$.

In particular, for all $m\geq 1$
\begin{equation}
    \gamma_c(G(1,m))=\left\lfloor \frac{m}{2}\right\rfloor+1
\end{equation}
\end{proposition}

\begin{proof}
Since $G(1,m)$ has only one row, the only applicable contamination rule is rule~(d).
In particular, the two end cells $(1,1)$ and $(1,m)$ each have only one neighbour, so they can
never become contaminated during the process unless they already belong to the initial set $S$.
Hence every contaminating set must contain both $(1,1)$ and $(1,m)$.

Once these two end cells are fixed, the cells $(1,2)$ and $(1,m-1)$ need not belong to $S$,
since they are contaminated by rule~(d) from the end cells whenever possible. The remaining
portion of the grid is the induced subgrid on the cells $(1,3),(1,4),\dots,(1,m-2)$,  which is isomorphic to $G(1,m-4)$. Therefore, for every $m\ge 4$,
\[
\gamma_c(G(1,m))=\gamma_c(G(1,m-4))+2.
\]
The initial values are immediate: $\gamma_c(G(1,1))=1$, and $\gamma_c(G(1,2))=\gamma_c(G(1,3))=2$.

We now prove the closed formula by induction on $m$.
The formula is easily verified for $m=1,2,3,4$. Let $m\ge 5$, and assume that
\[
\gamma_c(G(1,m-4))=\left\lfloor\frac{m-4}{2}\right\rfloor+1.
\]
Using the recurrence, we obtain
\[
\gamma_c(G(1,m))
=\gamma_c(G(1,m-4))+2
=\left\lfloor\frac{m-4}{2}\right\rfloor+1+2
=\left\lfloor\frac{m}{2}\right\rfloor+1.\qedhere
\]
\end{proof}

The following proposition improves upon the solution in Lemma~\ref{lem:zigzag-full}.

\begin{proposition}\label{prop:alternate-zigzag}
Let $m\ge n\ge 3$. Using the same zig-zag technique as in Lemma~\ref{lem:zigzag-full},
choose the contaminated cells alternately along the zig-zag path: the first cell is chosen,
the next is skipped, the next is chosen, and so on, except that the last cell of the zig-zag
path is always skipped. Additionally, the cell $(n,m)$ is always included in the initial
contaminating set $S$. Then the contamination process ends with a fully contaminated grid.
\end{proposition}

Figure~\ref{fig:examples_of_optimal_sets} illustrates four examples of initial contaminating sets constructed according to the process described in Proposition~\ref{prop:alternate-zigzag}. Moreover, the example given in Figure~\ref{fig:counter_example} shows an initial contaminating set $S$ constructed using Proposition~\ref{prop:alternate-zigzag}.

\begin{figure}[H]
    \centering

\begin{tikzpicture}[scale=0.8pt,x=0.75pt,y=0.75pt,yscale=-1,xscale=1]
%uncomment if require: \path (0,277); %set diagram left start at 0, and has height of 277

%Shape: Square [id:dp6543796289510028] 
\draw  [fill={rgb, 255:red, 208; green, 2; blue, 27 }  ,fill opacity=1 ] (0,0) -- (20,0) -- (20,20) -- (0,20) -- cycle ;
%Shape: Square [id:dp4231362180451913] 
\draw  [fill={rgb, 255:red, 255; green, 255; blue, 255 }  ,fill opacity=1 ] (20,20) -- (40,20) -- (40,40) -- (20,40) -- cycle ;
%Shape: Square [id:dp8928823321595718] 
\draw  [fill={rgb, 255:red, 255; green, 255; blue, 255 }  ,fill opacity=1 ] (20,20) -- (40,20) -- (40,40) -- (20,40) -- cycle ;
%Shape: Square [id:dp3544941984594737] 
\draw  [fill={rgb, 255:red, 208; green, 2; blue, 27 }  ,fill opacity=1 ] (40,40) -- (60,40) -- (60,60) -- (40,60) -- cycle ;
%Shape: Square [id:dp31862349952852076] 
\draw  [fill={rgb, 255:red, 255; green, 255; blue, 255 }  ,fill opacity=1 ] (60,60) -- (80,60) -- (80,80) -- (60,80) -- cycle ;
%Shape: Square [id:dp9230217549593234] 
\draw  [fill={rgb, 255:red, 255; green, 255; blue, 255 }  ,fill opacity=1 ] (20,0) -- (40,0) -- (40,20) -- (20,20) -- cycle ;
%Shape: Square [id:dp550199740739758] 
\draw  [fill={rgb, 255:red, 255; green, 255; blue, 255 }  ,fill opacity=1 ] (40,20) -- (60,20) -- (60,40) -- (40,40) -- cycle ;
%Shape: Square [id:dp9278932572207358] 
\draw  [fill={rgb, 255:red, 255; green, 255; blue, 255 }  ,fill opacity=1 ] (60,40) -- (80,40) -- (80,60) -- (60,60) -- cycle ;
%Shape: Square [id:dp9866503010749109] 
\draw  [fill={rgb, 255:red, 255; green, 255; blue, 255 }  ,fill opacity=1 ] (80,60) -- (100,60) -- (100,80) -- (80,80) -- cycle ;
%Shape: Square [id:dp23273930539494592] 
\draw  [fill={rgb, 255:red, 255; green, 255; blue, 255 }  ,fill opacity=1 ] (40,0) -- (60,0) -- (60,20) -- (40,20) -- cycle ;
%Shape: Square [id:dp647080491592527] 
\draw  [fill={rgb, 255:red, 255; green, 255; blue, 255 }  ,fill opacity=1 ] (60,20) -- (80,20) -- (80,40) -- (60,40) -- cycle ;
%Shape: Square [id:dp4067711840326531] 
\draw  [fill={rgb, 255:red, 208; green, 2; blue, 27 }  ,fill opacity=1 ] (80,40) -- (100,40) -- (100,60) -- (80,60) -- cycle ;
%Shape: Square [id:dp9828986633150887] 
\draw  [fill={rgb, 255:red, 255; green, 255; blue, 255 }  ,fill opacity=1 ] (60,0) -- (80,0) -- (80,20) -- (60,20) -- cycle ;
%Shape: Square [id:dp812418786928674] 
\draw  [fill={rgb, 255:red, 255; green, 255; blue, 255 }  ,fill opacity=1 ] (80,20) -- (100,20) -- (100,40) -- (80,40) -- cycle ;
%Shape: Square [id:dp6646691308333341] 
\draw  [fill={rgb, 255:red, 255; green, 255; blue, 255 }  ,fill opacity=1 ] (80,0) -- (100,0) -- (100,20) -- (80,20) -- cycle ;
%Shape: Square [id:dp7300050102852184] 
\draw  [fill={rgb, 255:red, 255; green, 255; blue, 255 }  ,fill opacity=1 ] (0,20) -- (20,20) -- (20,40) -- (0,40) -- cycle ;
%Shape: Square [id:dp6810570201491931] 
\draw  [fill={rgb, 255:red, 255; green, 255; blue, 255 }  ,fill opacity=1 ] (20,40) -- (40,40) -- (40,60) -- (20,60) -- cycle ;
%Shape: Square [id:dp649056603812767] 
\draw  [fill={rgb, 255:red, 255; green, 255; blue, 255 }  ,fill opacity=1 ] (40,60) -- (60,60) -- (60,80) -- (40,80) -- cycle ;
%Shape: Square [id:dp3991153904121272] 
\draw  [fill={rgb, 255:red, 255; green, 255; blue, 255 }  ,fill opacity=1 ] (0,40) -- (20,40) -- (20,60) -- (0,60) -- cycle ;
%Shape: Square [id:dp5010798873933362] 
\draw  [fill={rgb, 255:red, 255; green, 255; blue, 255 }  ,fill opacity=1 ] (0,60) -- (20,60) -- (20,80) -- (0,80) -- cycle ;
%Shape: Square [id:dp07562911778608239] 
\draw  [fill={rgb, 255:red, 255; green, 255; blue, 255 }  ,fill opacity=1 ] (20,60) -- (40,60) -- (40,80) -- (20,80) -- cycle ;
%Shape: Square [id:dp14821132502307566] 
\draw  [fill={rgb, 255:red, 255; green, 255; blue, 255 }  ,fill opacity=1 ] (100,20) -- (120,20) -- (120,40) -- (100,40) -- cycle ;
%Shape: Square [id:dp2962084071487161] 
\draw  [fill={rgb, 255:red, 255; green, 255; blue, 255 }  ,fill opacity=1 ] (100,0) -- (120,0) -- (120,20) -- (100,20) -- cycle ;
%Shape: Square [id:dp7898140718649509] 
\draw  [fill={rgb, 255:red, 255; green, 255; blue, 255 }  ,fill opacity=1 ] (100,40) -- (120,40) -- (120,60) -- (100,60) -- cycle ;
%Shape: Square [id:dp2087487825385841] 
\draw  [fill={rgb, 255:red, 208; green, 2; blue, 27 }  ,fill opacity=1 ] (100,60) -- (120,60) -- (120,80) -- (100,80) -- cycle ;

%Shape: Square [id:dp23603779396403446] 
\draw  [fill={rgb, 255:red, 208; green, 2; blue, 27 }  ,fill opacity=1 ] (150,0) -- (170,0) -- (170,20) -- (150,20) -- cycle ;
%Shape: Square [id:dp8776579981387996] 
\draw  [fill={rgb, 255:red, 255; green, 255; blue, 255 }  ,fill opacity=1 ] (170,20) -- (190,20) -- (190,40) -- (170,40) -- cycle ;
%Shape: Square [id:dp9296139147077696] 
\draw  [fill={rgb, 255:red, 255; green, 255; blue, 255 }  ,fill opacity=1 ] (170,20) -- (190,20) -- (190,40) -- (170,40) -- cycle ;
%Shape: Square [id:dp5597756514316697] 
\draw  [fill={rgb, 255:red, 208; green, 2; blue, 27 }  ,fill opacity=1 ] (190,40) -- (210,40) -- (210,60) -- (190,60) -- cycle ;
%Shape: Square [id:dp3843918155562881] 
\draw  [fill={rgb, 255:red, 255; green, 255; blue, 255 }  ,fill opacity=1 ] (210,60) -- (230,60) -- (230,80) -- (210,80) -- cycle ;
%Shape: Square [id:dp9968080613925641] 
\draw  [fill={rgb, 255:red, 255; green, 255; blue, 255 }  ,fill opacity=1 ] (170,0) -- (190,0) -- (190,20) -- (170,20) -- cycle ;
%Shape: Square [id:dp5473736596260255] 
\draw  [fill={rgb, 255:red, 255; green, 255; blue, 255 }  ,fill opacity=1 ] (190,20) -- (210,20) -- (210,40) -- (190,40) -- cycle ;
%Shape: Square [id:dp7775317179645495] 
\draw  [fill={rgb, 255:red, 255; green, 255; blue, 255 }  ,fill opacity=1 ] (210,40) -- (230,40) -- (230,60) -- (210,60) -- cycle ;
%Shape: Square [id:dp37982754796236395] 
\draw  [fill={rgb, 255:red, 255; green, 255; blue, 255 }  ,fill opacity=1 ] (230,40) -- (250,40) -- (250,60) -- (230,60) -- cycle ;
%Shape: Square [id:dp32475033543063103] 
\draw  [fill={rgb, 255:red, 255; green, 255; blue, 255 }  ,fill opacity=1 ] (190,0) -- (210,0) -- (210,20) -- (190,20) -- cycle ;
%Shape: Square [id:dp7975596076747736] 
\draw  [fill={rgb, 255:red, 255; green, 255; blue, 255 }  ,fill opacity=1 ] (210,20) -- (230,20) -- (230,40) -- (210,40) -- cycle ;
%Shape: Square [id:dp4947184336559327] 
\draw  [fill={rgb, 255:red, 208; green, 2; blue, 27 }  ,fill opacity=1 ] (230,60) -- (250,60) -- (250,80) -- (230,80) -- cycle ;
%Shape: Square [id:dp535933385722894] 
\draw  [fill={rgb, 255:red, 255; green, 255; blue, 255 }  ,fill opacity=1 ] (210,0) -- (230,0) -- (230,20) -- (210,20) -- cycle ;
%Shape: Square [id:dp2340811109720956] 
\draw  [fill={rgb, 255:red, 255; green, 255; blue, 255 }  ,fill opacity=1 ] (230,20) -- (250,20) -- (250,40) -- (230,40) -- cycle ;
%Shape: Square [id:dp7300325593335815] 
\draw  [fill={rgb, 255:red, 255; green, 255; blue, 255 }  ,fill opacity=1 ] (230,0) -- (250,0) -- (250,20) -- (230,20) -- cycle ;
%Shape: Square [id:dp8634289064625673] 
\draw  [fill={rgb, 255:red, 255; green, 255; blue, 255 }  ,fill opacity=1 ] (150,20) -- (170,20) -- (170,40) -- (150,40) -- cycle ;
%Shape: Square [id:dp577335783145466] 
\draw  [fill={rgb, 255:red, 255; green, 255; blue, 255 }  ,fill opacity=1 ] (170,40) -- (190,40) -- (190,60) -- (170,60) -- cycle ;
%Shape: Square [id:dp6967573147777177] 
\draw  [fill={rgb, 255:red, 255; green, 255; blue, 255 }  ,fill opacity=1 ] (190,60) -- (210,60) -- (210,80) -- (190,80) -- cycle ;
%Shape: Square [id:dp31643035102576267] 
\draw  [fill={rgb, 255:red, 255; green, 255; blue, 255 }  ,fill opacity=1 ] (150,40) -- (170,40) -- (170,60) -- (150,60) -- cycle ;
%Shape: Square [id:dp6984610612623455] 
\draw  [fill={rgb, 255:red, 255; green, 255; blue, 255 }  ,fill opacity=1 ] (150,60) -- (170,60) -- (170,80) -- (150,80) -- cycle ;
%Shape: Square [id:dp926800881712069] 
\draw  [fill={rgb, 255:red, 255; green, 255; blue, 255 }  ,fill opacity=1 ] (170,60) -- (190,60) -- (190,80) -- (170,80) -- cycle ;

%Shape: Square [id:dp19197284564101513] 
\draw  [fill={rgb, 255:red, 255; green, 255; blue, 255 }  ,fill opacity=1 ] (300,20) -- (320,20) -- (320,40) -- (300,40) -- cycle ;
%Shape: Square [id:dp23753787041171415] 
\draw  [fill={rgb, 255:red, 255; green, 255; blue, 255 }  ,fill opacity=1 ] (300,20) -- (320,20) -- (320,40) -- (300,40) -- cycle ;
%Shape: Square [id:dp31017735798859936] 
\draw  [fill={rgb, 255:red, 208; green, 2; blue, 27 }  ,fill opacity=1 ] (280,20) -- (300,20) -- (300,40) -- (280,40) -- cycle ;
%Shape: Square [id:dp7200370254974854] 
\draw  [fill={rgb, 255:red, 255; green, 255; blue, 255 }  ,fill opacity=1 ] (340,60) -- (360,60) -- (360,80) -- (340,80) -- cycle ;
%Shape: Square [id:dp9283010422939612] 
\draw  [fill={rgb, 255:red, 255; green, 255; blue, 255 }  ,fill opacity=1 ] (320,20) -- (340,20) -- (340,40) -- (320,40) -- cycle ;
%Shape: Square [id:dp5852435084563792] 
\draw  [fill={rgb, 255:red, 255; green, 255; blue, 255 }  ,fill opacity=1 ] (340,40) -- (360,40) -- (360,60) -- (340,60) -- cycle ;
%Shape: Square [id:dp11979518202394845] 
\draw  [fill={rgb, 255:red, 255; green, 255; blue, 255 }  ,fill opacity=1 ] (360,60) -- (380,60) -- (380,80) -- (360,80) -- cycle ;
%Shape: Square [id:dp4525513988743709] 
\draw  [fill={rgb, 255:red, 255; green, 255; blue, 255 }  ,fill opacity=1 ] (340,20) -- (360,20) -- (360,40) -- (340,40) -- cycle ;
%Shape: Square [id:dp1206743447239973] 
\draw  [fill={rgb, 255:red, 208; green, 2; blue, 27 }  ,fill opacity=1 ] (320,60) -- (340,60) -- (340,80) -- (320,80) -- cycle ;
%Shape: Square [id:dp6899636965709298] 
\draw  [fill={rgb, 255:red, 208; green, 2; blue, 27 }  ,fill opacity=1 ] (360,20) -- (380,20) -- (380,40) -- (360,40) -- cycle ;
%Shape: Square [id:dp7487920457717595] 
\draw  [fill={rgb, 255:red, 255; green, 255; blue, 255 }  ,fill opacity=1 ] (320,40) -- (340,40) -- (340,60) -- (320,60) -- cycle ;
%Shape: Square [id:dp43909177844840575] 
\draw  [fill={rgb, 255:red, 255; green, 255; blue, 255 }  ,fill opacity=1 ] (300,40) -- (320,40) -- (320,60) -- (300,60) -- cycle ;
%Shape: Square [id:dp6776599993531727] 
\draw  [fill={rgb, 255:red, 255; green, 255; blue, 255 }  ,fill opacity=1 ] (360,40) -- (380,40) -- (380,60) -- (360,60) -- cycle ;
%Shape: Square [id:dp9579311199044169] 
\draw  [fill={rgb, 255:red, 255; green, 255; blue, 255 }  ,fill opacity=1 ] (280,40) -- (300,40) -- (300,60) -- (280,60) -- cycle ;
%Shape: Square [id:dp466375075217218] 
\draw  [fill={rgb, 255:red, 255; green, 255; blue, 255 }  ,fill opacity=1 ] (280,60) -- (300,60) -- (300,80) -- (280,80) -- cycle ;
%Shape: Square [id:dp7468865825616577] 
\draw  [fill={rgb, 255:red, 255; green, 255; blue, 255 }  ,fill opacity=1 ] (300,60) -- (320,60) -- (320,80) -- (300,80) -- cycle ;
%Shape: Square [id:dp25673294339189745] 
\draw  [fill={rgb, 255:red, 255; green, 255; blue, 255 }  ,fill opacity=1 ] (380,20) -- (400,20) -- (400,40) -- (380,40) -- cycle ;
%Shape: Square [id:dp37680195176085574] 
\draw  [fill={rgb, 255:red, 255; green, 255; blue, 255 }  ,fill opacity=1 ] (380,40) -- (400,40) -- (400,60) -- (380,60) -- cycle ;
%Shape: Square [id:dp3713007265747179] 
\draw  [fill={rgb, 255:red, 208; green, 2; blue, 27 }  ,fill opacity=1 ] (380,60) -- (400,60) -- (400,80) -- (380,80) -- cycle ;

%Shape: Square [id:dp12088693262973194] 
\draw  [fill={rgb, 255:red, 255; green, 255; blue, 255 }  ,fill opacity=1 ] (450,20) -- (470,20) -- (470,40) -- (450,40) -- cycle ;
%Shape: Square [id:dp9161839180343685] 
\draw  [fill={rgb, 255:red, 255; green, 255; blue, 255 }  ,fill opacity=1 ] (450,20) -- (470,20) -- (470,40) -- (450,40) -- cycle ;
%Shape: Square [id:dp5737903303566102] 
\draw  [fill={rgb, 255:red, 208; green, 2; blue, 27 }  ,fill opacity=1 ] (430,20) -- (450,20) -- (450,40) -- (430,40) -- cycle ;
%Shape: Square [id:dp6117966118072526] 
\draw  [fill={rgb, 255:red, 255; green, 255; blue, 255 }  ,fill opacity=1 ] (490,60) -- (510,60) -- (510,80) -- (490,80) -- cycle ;
%Shape: Square [id:dp38723858835006586] 
\draw  [fill={rgb, 255:red, 255; green, 255; blue, 255 }  ,fill opacity=1 ] (470,20) -- (490,20) -- (490,40) -- (470,40) -- cycle ;
%Shape: Square [id:dp9464367719873187] 
\draw  [fill={rgb, 255:red, 255; green, 255; blue, 255 }  ,fill opacity=1 ] (490,40) -- (510,40) -- (510,60) -- (490,60) -- cycle ;
%Shape: Square [id:dp15514691491227217] 
\draw  [fill={rgb, 255:red, 255; green, 255; blue, 255 }  ,fill opacity=1 ] (510,20) -- (530,20) -- (530,40) -- (510,40) -- cycle ;
%Shape: Square [id:dp9744172449809604] 
\draw  [fill={rgb, 255:red, 255; green, 255; blue, 255 }  ,fill opacity=1 ] (490,20) -- (510,20) -- (510,40) -- (490,40) -- cycle ;
%Shape: Square [id:dp5378716327063273] 
\draw  [fill={rgb, 255:red, 208; green, 2; blue, 27 }  ,fill opacity=1 ] (470,60) -- (490,60) -- (490,80) -- (470,80) -- cycle ;
%Shape: Square [id:dp8509285793778751] 
\draw  [fill={rgb, 255:red, 208; green, 2; blue, 27 }  ,fill opacity=1 ] (510,60) -- (530,60) -- (530,80) -- (510,80) -- cycle ;
%Shape: Square [id:dp7571321141667884] 
\draw  [fill={rgb, 255:red, 255; green, 255; blue, 255 }  ,fill opacity=1 ] (470,40) -- (490,40) -- (490,60) -- (470,60) -- cycle ;
%Shape: Square [id:dp4539344218935295] 
\draw  [fill={rgb, 255:red, 255; green, 255; blue, 255 }  ,fill opacity=1 ] (450,40) -- (470,40) -- (470,60) -- (450,60) -- cycle ;
%Shape: Square [id:dp6097076154332898] 
\draw  [fill={rgb, 255:red, 255; green, 255; blue, 255 }  ,fill opacity=1 ] (510,40) -- (530,40) -- (530,60) -- (510,60) -- cycle ;
%Shape: Square [id:dp5806201152769328] 
\draw  [fill={rgb, 255:red, 255; green, 255; blue, 255 }  ,fill opacity=1 ] (430,40) -- (450,40) -- (450,60) -- (430,60) -- cycle ;
%Shape: Square [id:dp7450371692447985] 
\draw  [fill={rgb, 255:red, 255; green, 255; blue, 255 }  ,fill opacity=1 ] (430,60) -- (450,60) -- (450,80) -- (430,80) -- cycle ;
%Shape: Square [id:dp10652005785477081] 
\draw  [fill={rgb, 255:red, 255; green, 255; blue, 255 }  ,fill opacity=1 ] (450,60) -- (470,60) -- (470,80) -- (450,80) -- cycle ;

% Text Node
\draw (10,103) node [anchor=north west][inner sep=0.75pt]  [font=\footnotesize] [align=left] {$\displaystyle n$ and $\displaystyle m$ even};
% Text Node
\draw (145,103) node [anchor=north west][inner sep=0.75pt]  [font=\footnotesize] [align=left] {$\displaystyle n$ even, $\displaystyle m$ odd};
% Text Node
\draw (285,103) node [anchor=north west][inner sep=0.75pt]  [font=\footnotesize] [align=left] {$\displaystyle n$ odd, $\displaystyle m$ even};
% Text Node
\draw (428,103) node [anchor=north west][inner sep=0.75pt]  [font=\footnotesize] [align=left] {$\displaystyle n$ odd, $\displaystyle m$ odd};

\end{tikzpicture}

    \caption{Examples of initial contaminating sets $S$ constructed according to Proposition~\ref{prop:alternate-zigzag}.
}
    \label{fig:examples_of_optimal_sets}
\end{figure}

\begin{corollary}\label{corollary_upper_bound}
    For all integers $m \geq n \geq 1$, we have 
    \begin{equation}
        \gamma_c (G(n,m)) \leq \left\lfloor \frac{m}{2} \right\rfloor + 1.
    \end{equation}
\end{corollary}

\begin{proof}
    In the construction of $S$ given in Proposition~\ref{prop:alternate-zigzag}, 
    the cells of the zig-zag path are chosen alternately, giving 
    $\left\lfloor \frac{m}{2} \right\rfloor$ cells. Adding the cell $(n,m)$ yields 
    $\left\lfloor \frac{m}{2} \right\rfloor + 1$ cells in total, establishing the bound.\qedhere
\end{proof}

\begin{proposition}\label{proposition_lower_bound}
For all integers $m \ge n \ge 1$, we have
\begin{equation}
    \gamma_c(G(n,m)) \ge \left\lfloor \frac{m}{2} \right\rfloor + 1.
\end{equation}
\end{proposition}

\begin{proof}
Let $S$ be a contaminating set of minimum cardinality, so that $|S|=\gamma_c(G(n,m))$.

We say that a column is \emph{nonempty} if it contains at least one cell of $S$, and \emph{empty}
otherwise.

By Corollary~\ref{coro:boundary}, the first and the last columns must be nonempty. Moreover, by
Lemma~\ref{lem:no-two-empty}, no two consecutive columns can both
be empty, since otherwise the contamination process cannot produce a fully contaminated grid. We therefore distinguish two cases.

We start with the case $m=2k$ is even. 
Since the first and last columns are nonempty, it remains to consider the $2k-2$ interior columns
$2,3,\dots,2k-1$. Partition them into the $k-1$ consecutive pairs $(2,3),\ (4,5),\ \dots,\ (2k-2,2k-1)$.
 
In each such pair, at least one column must be nonempty; otherwise we would have two consecutive
empty columns, contradicting
Lemma~\ref{lem:no-two-empty}. Hence there are at least $k-1$
nonempty interior columns. Adding the first and last columns, we obtain at least $(k-1)+2 = k+1$ 
nonempty columns in total. Since each nonempty column contains at least one cell of $S$, we get
\[
|S| \ge k+1 = \frac{m}{2}+1 = \left\lfloor \frac{m}{2} \right\rfloor +1.
\]

We now turn to the case $m=2k+1$ is odd. 
Partition the first $2k$ columns into the $k$ consecutive pairs $(1,2),\ (3,4),\ \dots,\ (2k-1,2k)$.
 
Again, each pair must contain at least one nonempty column, for otherwise two consecutive columns
would be empty, contradicting
Lemma~\ref{lem:no-two-empty}. In addition, the last column
$2k+1$ must be nonempty by Corollary~\ref{coro:boundary}. Therefore there are at least $k+1$ nonempty columns in total. Since each nonempty column contains at least one cell of $S$, it follows
that
\[
|S| \ge k+1 = \left\lfloor \frac{2k+1}{2} \right\rfloor +1
= \left\lfloor \frac{m}{2} \right\rfloor +1.
\]

In both cases,
\[
\gamma_c(G(n,m)) = |S| \ge \left\lfloor \frac{m}{2} \right\rfloor +1.\qedhere
\]
\end{proof}

\begin{lemma}\label{lem:canonical-odd-columns}
Let $m=2k+1$ be odd and let $n\ge 3$. If $S$ is an optimal contamination set of $G(n,m)$ and $|S|=k+1$, then:
\begin{enumerate}
    \item each odd column contains exactly one cell of $S$;
    \item each even column contains no cell of $S$.
\end{enumerate}
In particular, every optimal solution is supported on the odd columns
$1,3,5,\dots,2k+1$.
\end{lemma}

\begin{proof}
By Corollary~\ref{coro:boundary} the first and last columns must be nonempty, and by Lemma~\ref{lem:no-two-empty} no two consecutive columns can both be empty. Therefore, among the $2k+1$ columns, the set of nonempty columns must form a binary pattern of length $2k+1$ that: starts with a nonempty column, ends with a nonempty column, and contains no two consecutive empty columns.

Such a pattern has at least $k+1$ nonempty columns, and equality is attained only by the alternating pattern
$1,0,1,0,\dots,1$, 
that is, precisely the odd columns are nonempty.

Since $S$ contains exactly $k+1$ cells in total, each nonempty column contains exactly one cell of $S$. Hence every odd column contains exactly one contaminated cell and every even column contains none.\qedhere
\end{proof}

\subsection{The exact formula and its consequences}
Combining the previous bounds, we obtain the exact value of the power contamination number. We then record several direct consequences, including recurrence relations and structural properties of optimal contamination sets in the odd-width case.

\begin{theorem}\label{theo:main}
    For all integers $m \geq n \geq 1$, we have 
    \begin{equation}
        \gamma_c(G(n,m)) =
        \begin{cases}
            \left\lceil \dfrac{m}{2} \right\rceil + 1, & \text{if $n=2$ and $m$ is odd}, \\[0.6em]
            \left\lfloor \dfrac{m}{2} \right\rfloor + 1, & \text{otherwise}.
        \end{cases}
    \end{equation}
\end{theorem}

\begin{proof}
    If $m \geq n \geq 3$, the result follows immediately from Proposition~\ref{proposition_lower_bound} 
    and Corollary~\ref{corollary_upper_bound}.

    Now suppose $n = 2$. If $m$ is even, the construction in 
    Proposition~\ref{prop:alternate-zigzag} gives a set $S$ 
    of size $\lfloor \frac{m}{2} \rfloor + 1$ which leads to full contamination, proving the bound.

    It remains to treat the case $n=2$ and $m$ odd. By Corollary~\ref{coro:boundary}, 
    each boundary of $G(2,m)$ must contain at least one contaminated cell. Without loss of generality, 
    take $(1,1)$ and $(2,m)$ in $S$. For contamination to propagate from $(1,1)$, 
    the closest additional contaminated cell in the top row must be $(1,3)$: 
    if instead $(2,3)$ were chosen, none of rules~\textup{(\texttt{a})}--\textup{(\texttt{h})} would apply to infect $(1,2)$. 
    Repeating this argument forces all cells $(1,2i+1)$ for 
    $i=0,1,\dots,\frac{m-3}{2}$ to be in $S$.

    After these are chosen, $(1,m-2)$ becomes contaminated. However, its rightmost 
    contaminated neighbour in $S$ is $(2,m)$, and the gap between them prevents 
    contamination from completing unless at least one cell among 
    $(1,m-1), (1,m), (2,m-2), (2,m-1)$ is in $S$. Without loss of generality, 
    take $(1,m)$.

    Therefore, the minimal $S$ is
    \[
        S = \{ (1,2i+1) \mid i = 0,1,\dots,\tfrac{m-1}{2} \} \cup \{ (2,m) \},
    \]
    which has size $\frac{m+1}{2} + 1 = \lceil \frac{m}{2} \rceil + 1$. This completes the proof.\qedhere
\end{proof}

\begin{corollary}  
For all integers $m > n \geq 3$, we have  
\begin{equation}\label{rec_2}
    \gamma_c(G(n,m)) = \gamma_c(G(n,m-1)) + \frac{1}{2}\left( 1 + (-1)^m \right).
\end{equation}
\end{corollary}
\begin{proof}
    This follows directly from Theorem~\ref{theo:main}.\qedhere
\end{proof}

\begin{corollary}
Let $m \geq n \geq 3$, $p \leq n-3$, and $q \leq m-3$, with $m-q \geq n-p$. Then  
\begin{equation}
    \gamma_c(G(n,m)) =
    \begin{cases}
        \gamma_c(G(n-p,m-q)) + \left\lceil \dfrac{q}{2} \right\rceil, & \text{if $m$ is even and $q$ is odd}, \\[8pt]
        \gamma_c(G(n-p,m-q)) + \left\lfloor \dfrac{q}{2} \right\rfloor, & \text{otherwise}.
    \end{cases}
\end{equation}
\end{corollary}

\begin{proof}
Since $p\le n-3$, we have $n-p\ge 3$, so by Theorem~\ref{theo:main},
\[
\gamma_c(G(n,m))=\left\lfloor\frac{m}{2}\right\rfloor+1
=\gamma_c(G(n-p,m)).
\]
Hence
\[
\gamma_c(G(n,m))
=\gamma_c(G(n-p,m-q))+\sum_{i=0}^{q-1}\frac12\bigl(1+(-1)^{m-i}\bigr).\qedhere
\]
\end{proof}

\section{Enumeration and Structure of Optimal Solutions}\label{sec:enumeration}

Having determined the exact value of \(\gamma_c(G(n,m))\), we now turn to the enumeration and structural description of optimal contamination sets. Our goal in this section is twofold: first, to compute or characterize the number \(\alpha_{n,m}\) of optimal solutions for several families of grids, and second, to highlight the combinatorial structures that arise from these optimal configurations.

Table~\ref{tab:number_optimal_solutons} presents the initial values of \(\alpha_{n,m}\) for
\(m \ge n \ge 1\). These values were obtained by an exact search procedure: for each grid
\(G(n,m)\), we systematically examined all configurations of
\(\gamma_c(G(n,m))\) initially contaminated cells among the \(nm\) cells of the grid, that is,
all \(\binom{nm}{\gamma_c(G(n,m))}\) candidate configurations, and counted those that lead to
full contamination.

As part of this exact computational study, we determined the number of optimal contamination
configurations \(\alpha_{n,m}\) for all grids \(G(n,m)\) with \(1 \le n \le m \le 9\). Interestingly,
several of the resulting sequences coincide with well-known entries in the OEIS \cite{OEIS}.

\begin{table}[H]
\centering
\begin{tabular*}{\linewidth}{@{\extracolsep{\fill}}c|llllllllll@{}}
\toprule
$n\backslash m$ &  1 & 2 & 3  & 4  & 5  & 6   & 7   & 8    & 9   & $\cdots$                           \\ \midrule

1                                                           & 1 & 1 & 1  & 2  & 1  & 3   & 1   & 4    & 1   & $\cdots$                           \\
2                                                           &   & 2 & 10 & 8  & 48 & 24  & 176 & 64   & 560 & $\cdots$\\
3                                                           &   &   & 2  & 20 & 10 & 130 & 40  & 640  & 144 & $\cdots$                           \\
4                                                           &   &   &    & 12 & 8  & 210 & 68  & 1736 & 412 & $\cdots$                           \\
5                                                           &   &   &    &    & 6  & 232 & 88  & 3048 & 786 & $\cdots$                           \\
6                                                           &   &   &    &    &    & 122 & 48  & 3104 & 820 & $\cdots$                           \\
7                                                           &   &   &    &    &    &     & 22  & 2260 & 644 & $\cdots$                           \\
8                                                           &   &   &    &    &    &     &     & 912  & 272 & $\cdots$                           \\
9                                                           &   &   &    &    &    &     &     &      & 90  & $\cdots$                           \\
$\vdots$                                                    &   &   &    &    &    &     &     &      &     & $\ddots$                          
\end{tabular*}
\caption{The first values of $\alpha_{n,m}$ for $m\geq n \geq 1$.}
\label{tab:number_optimal_solutons}
\end{table}

\subsection{The families \texorpdfstring{$G(1,m)$}{} and \texorpdfstring{$G(2,m)$}{}}

\begin{proposition}\label{proposition_nmbr_opt_solution}
For all integers $m \geq 1$, we have
\begin{equation}
        \alpha_{1,m} =
    \begin{cases}
        1, & \text{if $m$ is odd}, \\[4pt]
        \dfrac{m}{2}, & \text{if $m$ is even}.
    \end{cases}
\end{equation}
\end{proposition}

\begin{proof}
If $m$ is odd, there is only one optimal solution: the one where the cells in odd positions are contaminated and those in even positions are not.

If $m$ is even, we observe that the leftmost and rightmost cells must both be contaminated. Therefore, we need to determine the positions of the remaining $\gamma_c(G(1,m)) - 2$ contaminated cells among the $m - 2$ remaining positions, such that no two uncontaminated cells are adjacent.

Since $m$ is even, we have
\[
    \gamma_c(G(1,m)) = \left\lfloor \frac{m}{2} \right\rfloor + 1 = \frac{m}{2} + 1.
\]
The number of uncontaminated cells is then $m - \gamma_c(G(1,m)) = \frac{m}{2} - 1$.

The count of ways to choose these uncontaminated cells with no two adjacent among the $m - 2$ available positions is
\[
    \binom{(m - 2) - \left( \frac{m}{2} - 1 \right) + 1}{\frac{m}{2} - 1}
    = \binom{\frac{m}{2}}{\frac{m}{2} - 1}
    = \frac{m}{2}.
\]
This completes the proof.\qedhere
\end{proof}

\begin{proposition}\label{prop:a_2_2k}
For every integer \(k \ge 1\), we have
\begin{equation}
    \alpha_{2,2k}=k 2^{k}.
\end{equation}
Equivalently, the sequence \((\alpha_{2,2k})_{k\ge 1}\) coincides with \textup{\texttt{A036289}}.
\end{proposition}

\begin{proof}
    Consider a grid of $n=2$ rows and $m=2k$ columns, with $k \geq 1$.  
    We call a column \emph{clean} if it contains no contaminated cells.  

    From Lemma~\ref{lem:no-two-empty}, we know that no pair of consecutive columns can be fully clean in an optimal distribution of contaminated cells. Moreover, since $\gamma_c(G(2,2k)) = k+1$, there must be exactly $2k - (k+1) = k-1$ clean columns in an optimal distribution.  

    By Proposition~\ref{prop:rectangular-boundary}(C), the first and the last columns cannot be clean. Thus, the only possible clean columns are  columns $2,3,\dots,2k-1$.  

    Choosing $k-1$ non-consecutive clean columns from these $2k-2$ candidates is equivalent to counting the number of words of length $2k-2$ over the alphabet $\{a,b\}$ with exactly $k-1$ occurrences of $a$ and no two $a$'s consecutive.  
    This number is given by
    \[
        \binom{(2k-2) - (k-1) + 1}{k-1} 
        = \binom{k}{k-1} 
        = k.
    \]

    On the other hand, since there are $k-1$ non-consecutive clean columns, there must be exactly one pair of consecutive contaminated columns among the remaining $k+1$ columns.  

    Each of these two consecutive contaminated columns contains exactly one contaminated cell. By Proposition~\ref{prop:rectangular-boundary}(C), both the upper and the lower row of the grid must contain at least one contaminated cell. Therefore, the two contaminated cells must be placed diagonally (upper cell in the left column and lower cell in the right column, or vice versa). This yields two possible configurations.

    For each of the remaining $k-1$ contaminated columns, the contaminated cell can be chosen arbitrarily (either the upper or the lower cell), giving $2^{k-1}$ possibilities.  

    Consequently, the total number of optimal distributions is
    \[
        \alpha_{2,2k} = k \cdot 2 \cdot 2^{k-1} = k \cdot 2^{k}.\qedhere
    \]
\end{proof}

\begin{corollary}
    For all integers $k \geq 1$, the sequence $\alpha_{2,2k}$ satisfies the recurrence
    \begin{equation} 
        \alpha_{2,2k} = 2 \alpha_{2,2k-2} + 2^k,
    \end{equation}
    with the convention $\alpha_{2,0} = 0$.
\end{corollary}

\begin{proof}
    This follows directly from Proposition~\ref{prop:a_2_2k}.\qedhere
\end{proof}

\begin{lemma}\label{lem:nbr_clean_cols_odd_m}
    Consider a grid of $n=2$ rows and $m=2k+1$ columns, with $k \geq 0$.  
    In any optimal solution, the number of clean columns is   either $k$  or $k-1$ clean columns. 
\end{lemma}

\begin{proof}
    Let us consider a grid of $n=2$ rows and $m=2k+1$ columns, with $k \geq 1$.  
    We have $\gamma_c(G(2,2k+1)) = k+2$.

    By Proposition~\ref{prop:rectangular-boundary}(C), the first and last columns cannot be clean, and hence they must each contain at least one contaminated cell. Therefore, $2$ contaminated cells are already fixed, leaving
    $(k+2) - 2 = k$ contaminated cells to be placed among the remaining columns.

    From Lemma~\ref{lem:no-two-empty}, no pair of consecutive columns can be fully clean in an optimal distribution.  
    Since the first and last columns are certainly contaminated, we only need to ensure that among the  columns $2,3,\dots,2k$, no two consecutive columns are simultaneously clean.  
    There are $2k-1$ such  columns.  

    The minimum number of contaminated columns required to guarantee that no two consecutive  columns are clean is $k-1$.  
    Thus, up to this point, we have used  $2 + (k-1) = k+1$ contaminated cells, which means that one additional contaminated cell remains to be placed. We call this extra contaminated cell the \emph{free cell}.  

    Depending on the position of this free cell, the final number of clean columns will be either $k$ or $k-1$.  Figure~\ref{fig:two_sol_m_equal_7} illustrates these two cases for $m=7$, where $\gamma_c(G(2,7)) = 5$ and $k=3$.\qedhere
\end{proof}

\begin{figure}[H]
    \centering

\begin{tikzpicture}[scale=0.8pt,x=0.75pt,y=0.75pt,yscale=-1,xscale=1]
%uncomment if require: \path (0,473); %set diagram left start at 0, and has height of 473

%Shape: Square [id:dp49359536171818275] 
\draw  [fill={rgb, 255:red, 208; green, 2; blue, 27 }  ,fill opacity=1 ] (157.5,39.5) -- (177.5,39.5) -- (177.5,59.5) -- (157.5,59.5) -- cycle ;
%Shape: Square [id:dp29218015675379516] 
\draw  [fill={rgb, 255:red, 255; green, 255; blue, 255 }  ,fill opacity=1 ] (177.5,39.5) -- (197.5,39.5) -- (197.5,59.5) -- (177.5,59.5) -- cycle ;
%Shape: Square [id:dp3830045864971894] 
\draw  [fill={rgb, 255:red, 255; green, 255; blue, 255 }  ,fill opacity=1 ] (177.5,59.5) -- (197.5,59.5) -- (197.5,79.5) -- (177.5,79.5) -- cycle ;
%Shape: Square [id:dp27219872802018275] 
\draw  [fill={rgb, 255:red, 255; green, 255; blue, 255 }  ,fill opacity=1 ] (157.5,59.5) -- (177.5,59.5) -- (177.5,79.5) -- (157.5,79.5) -- cycle ;
%Shape: Square [id:dp47295336364400264] 
\draw  [fill={rgb, 255:red, 208; green, 2; blue, 27 }  ,fill opacity=1 ] (197.5,39.5) -- (217.5,39.5) -- (217.5,59.5) -- (197.5,59.5) -- cycle ;
%Shape: Square [id:dp05779638255995856] 
\draw  [fill={rgb, 255:red, 255; green, 255; blue, 255 }  ,fill opacity=1 ] (197.5,59.5) -- (217.5,59.5) -- (217.5,79.5) -- (197.5,79.5) -- cycle ;
%Shape: Square [id:dp40022481311074687] 
\draw  [fill={rgb, 255:red, 255; green, 255; blue, 255 }  ,fill opacity=1 ] (217.5,59.5) -- (237.5,59.5) -- (237.5,79.5) -- (217.5,79.5) -- cycle ;
%Shape: Square [id:dp8066803991026978] 
\draw  [fill={rgb, 255:red, 255; green, 255; blue, 255 }  ,fill opacity=1 ] (217.5,39.5) -- (237.5,39.5) -- (237.5,59.5) -- (217.5,59.5) -- cycle ;
%Shape: Square [id:dp3016450283166059] 
\draw  [fill={rgb, 255:red, 255; green, 255; blue, 255 }  ,fill opacity=1 ] (237.5,59.5) -- (257.5,59.5) -- (257.5,79.5) -- (237.5,79.5) -- cycle ;
%Shape: Square [id:dp47729909119155933] 
\draw  [fill={rgb, 255:red, 208; green, 2; blue, 27 }  ,fill opacity=1 ] (237.5,39.5) -- (257.5,39.5) -- (257.5,59.5) -- (237.5,59.5) -- cycle ;
%Shape: Square [id:dp869943131569602] 
\draw  [fill={rgb, 255:red, 255; green, 255; blue, 255 }  ,fill opacity=1 ] (257.5,59.5) -- (277.5,59.5) -- (277.5,79.5) -- (257.5,79.5) -- cycle ;
%Shape: Square [id:dp15050989483155286] 
\draw  [fill={rgb, 255:red, 255; green, 255; blue, 255 }  ,fill opacity=1 ] (257.5,39.5) -- (277.5,39.5) -- (277.5,59.5) -- (257.5,59.5) -- cycle ;
%Shape: Square [id:dp9523612316675725] 
\draw  [fill={rgb, 255:red, 208; green, 2; blue, 27 }  ,fill opacity=1 ] (277.5,59.5) -- (297.5,59.5) -- (297.5,79.5) -- (277.5,79.5) -- cycle ;
%Shape: Square [id:dp7038352942995891] 
\draw  [fill={rgb, 255:red, 208; green, 2; blue, 27 }  ,fill opacity=1 ] (277.5,39.5) -- (297.5,39.5) -- (297.5,59.5) -- (277.5,59.5) -- cycle ;

%Shape: Square [id:dp6684793899520203] 
\draw  [fill={rgb, 255:red, 208; green, 2; blue, 27 }  ,fill opacity=1 ] (354,39.5) -- (374,39.5) -- (374,59.5) -- (354,59.5) -- cycle ;
%Shape: Square [id:dp8710401842207174] 
\draw  [fill={rgb, 255:red, 255; green, 255; blue, 255 }  ,fill opacity=1 ] (374,39.5) -- (394,39.5) -- (394,59.5) -- (374,59.5) -- cycle ;
%Shape: Square [id:dp024427348569463714] 
\draw  [fill={rgb, 255:red, 255; green, 255; blue, 255 }  ,fill opacity=1 ] (374,59.5) -- (394,59.5) -- (394,79.5) -- (374,79.5) -- cycle ;
%Shape: Square [id:dp03189335045844843] 
\draw  [fill={rgb, 255:red, 255; green, 255; blue, 255 }  ,fill opacity=1 ] (354,59.5) -- (374,59.5) -- (374,79.5) -- (354,79.5) -- cycle ;
%Shape: Square [id:dp6145458906977503] 
\draw  [fill={rgb, 255:red, 208; green, 2; blue, 27 }  ,fill opacity=1 ] (394,39.5) -- (414,39.5) -- (414,59.5) -- (394,59.5) -- cycle ;
%Shape: Square [id:dp5589512709701421] 
\draw  [fill={rgb, 255:red, 255; green, 255; blue, 255 }  ,fill opacity=1 ] (394,59.5) -- (414,59.5) -- (414,79.5) -- (394,79.5) -- cycle ;
%Shape: Square [id:dp19617189489841502] 
\draw  [fill={rgb, 255:red, 255; green, 255; blue, 255 }  ,fill opacity=1 ] (414,59.5) -- (434,59.5) -- (434,79.5) -- (414,79.5) -- cycle ;
%Shape: Square [id:dp18819684396073866] 
\draw  [fill={rgb, 255:red, 255; green, 255; blue, 255 }  ,fill opacity=1 ] (414,39.5) -- (434,39.5) -- (434,59.5) -- (414,59.5) -- cycle ;
%Shape: Square [id:dp8114132646828676] 
\draw  [fill={rgb, 255:red, 255; green, 255; blue, 255 }  ,fill opacity=1 ] (434,59.5) -- (454,59.5) -- (454,79.5) -- (434,79.5) -- cycle ;
%Shape: Square [id:dp14096663683347388] 
\draw  [fill={rgb, 255:red, 208; green, 2; blue, 27 }  ,fill opacity=1 ] (434,39.5) -- (454,39.5) -- (454,59.5) -- (434,59.5) -- cycle ;
%Shape: Square [id:dp9211338760246686] 
\draw  [fill={rgb, 255:red, 255; green, 255; blue, 255 }  ,fill opacity=1 ] (474,59.5) -- (494,59.5) -- (494,79.5) -- (474,79.5) -- cycle ;
%Shape: Square [id:dp6937914482645808] 
\draw  [fill={rgb, 255:red, 255; green, 255; blue, 255 }  ,fill opacity=1 ] (454,39.5) -- (474,39.5) -- (474,59.5) -- (454,59.5) -- cycle ;
%Shape: Square [id:dp16518991445435272] 
\draw  [fill={rgb, 255:red, 208; green, 2; blue, 27 }  ,fill opacity=1 ] (454,59.5) -- (474,59.5) -- (474,79.5) -- (454,79.5) -- cycle ;
%Shape: Square [id:dp0769075790408984] 
\draw  [fill={rgb, 255:red, 208; green, 2; blue, 27 }  ,fill opacity=1 ] (474,39.5) -- (494,39.5) -- (494,59.5) -- (474,59.5) -- cycle ;

% Text Node
\draw (170,92.33) node [anchor=north west][inner sep=0.75pt]  [font=\footnotesize] [align=left] {$\displaystyle 3$ clean columns};
% Text Node
\draw (365,92.33) node [anchor=north west][inner sep=0.75pt]  [font=\footnotesize] [align=left] {$\displaystyle 2$ clean columns};

\end{tikzpicture}
    \caption{Two optimal solutions for $m=7$.}
    \label{fig:two_sol_m_equal_7}
\end{figure}

\begin{proposition}\label{prop:nbr_optimal_solutions_whit_k_clean_cols}
    Consider a grid of $n=2$ rows and $m=2k+1$ columns, with $k \geq 1$.  
    Then, the number of optimal solutions with exactly $k$ clean columns is $(k+1)   2^{k}$.
\end{proposition}

\begin{proof}
    Suppose first that the optimal solution contains $k$ clean columns.  
    Similarly to the proof of Proposition~\ref{prop:a_2_2k}, the number of ways of choosing $k$ non-consecutive clean columns among the $2k-1$ available  columns is
    \[
        \binom{(2k+1-2) - k + 1}{k} 
        = \binom{k}{k} 
        = 1.
    \]

    In this configuration, the \emph{free cell} can only be placed in one of the $k+1$ already contaminated columns. Hence, there are $k+1$ possible positions for the free cell.  
    The column containing the free cell becomes fully contaminated (both upper and lower cells are contaminated).

    Among the remaining $k$ contaminated columns (those with exactly one contaminated cell), each can be chosen freely in the upper or lower row. Since the fully contaminated column already guarantees the condition of Proposition~\ref{prop:rectangular-boundary}(C), the other $k$ columns can be chosen independently, giving $2^{k}$ possibilities.

    Consequently, the number of optimal solutions in this case is $1 \cdot (k+1) \cdot 2^{k} = (k+1)2^{k}$.\qedhere
\end{proof}

For every integer $k\ge 1$, let $\beta_k$ denote the number of optimal solutions of $G(2,2k+1)$
having exactly $k-1$ clean columns.

\begin{corollary}\label{coro:alpha2odd-decomposition}
For every integer $k\ge 1$, we have
\begin{equation}
    \alpha_{2,2k+1}=(k+1)2^k+\beta_k.
\end{equation}
\end{corollary}

\begin{proof}
By Lemma~\ref{lem:nbr_clean_cols_odd_m}, every optimal solution of $G(2,2k+1)$ has either exactly $k$ clean columns
or exactly $k-1$ clean columns. By Proposition~\ref{prop:nbr_optimal_solutions_whit_k_clean_cols}, the number of optimal solutions with exactly
$k$ clean columns is $(k+1)2^k$. By definition, the number of optimal solutions with exactly
$k-1$ clean columns is $\beta_k$. Summing the two disjoint cases gives the result. 
\end{proof}

The remaining case, namely the enumeration of optimal solutions with exactly $k-1$ clean columns,
seems to be more delicate. Our computations suggest that the total number $\alpha_{2,2k+1}$
coincides with the OEIS sequence \textup{\texttt{A084857}}. This leads to the following conjecture.
\begin{conjecture}\label{conj:alpha2odd-total}
For every integer \(k \ge 1\), we have
\begin{equation}
    \alpha_{2,2k+1}=(k+1)(3k+2)2^{k-1}.
\end{equation}
Equivalently, the sequence \((\alpha_{2,2k+1})_{k\ge 1}\) coincides with \textup{\texttt{A084857}} shifted by one index, namely
\[
\alpha_{2,2k+1}=\textup{\texttt{A084857}}(k+1).
\]
\end{conjecture}

Assuming Conjecture~\ref{conj:alpha2odd-total}, we immediately obtain the following expression
for the more difficult subfamily counted by $\beta_k$.

\begin{observation} 
If Conjecture~\ref{conj:alpha2odd-total} holds, then by Corollary~\ref{coro:alpha2odd-decomposition}, for every integer $k\ge 1$,
\[
\beta_k=\alpha_{2,2k+1}-(k+1)2^k=(k+1)(3k+2)2^{k-1}-(k+1)2^k
      =(k+1)(3k)2^{k-1}
      =3k(k+1)2^{k-1}.
\]
 \end{observation}

\subsection{The family \texorpdfstring{$G(3,2k+1)$}{} and ternary words}

We now specialize to the case of optimal solutions on \(G(3,2k+1)\).

By Lemma~\ref{lem:canonical-odd-columns}, every optimal contamination set
\(S\) of \(G(3,2k+1)\) has exactly one contaminated cell in each odd column and
no contaminated cell in any even column. Hence we may write
\[
S=\{(r_j,2j+1)\mid 0\le j\le k\}, 
\]
 with $r_j\in\{1,2,3\}$, and encode \(S\) by the word $w(S)=r_0r_1\cdots r_k\in\{1,2,3\}^{k+1}$. 

When no confusion is possible, we simply write \(w\).

For such a set \(S\), define
\[
D(S):=S\cup \{(r_j,2j)\mid 1\le j\le k,\ r_{j-1}=r_j\}.
\]
Thus \(D(S)\) is obtained from \(S\) by adding exactly those even-column cells
that are immediately contaminated by rule~\textup{(\texttt{d})} from two equal adjacent
letters of \(w\).

\begin{lemma}\label{lem:3row-avoid}
Let \(S\) be an optimal contamination set of \(G(3,2k+1)\), and let
\(w=r_0r_1\cdots r_k\in\{1,2,3\}^{k+1}\) be its associated word. If \(w\)
avoids both factors \(13\) and \(31\), then
\[
C(S)=D(S).
\]
In particular, each column of \(C(S)\) contains at most one contaminated cell,
and therefore \(S\) does not fully contaminate \(G(3,2k+1)\).
\end{lemma}

\begin{proof}
We first prove that \(D(S)\subseteq C(S)\). Every cell of \(S\) is contaminated
initially. Now let \(1\le j\le k\) and assume that \(r_{j-1}=r_j\). Then the
two contaminated cells $(r_{j-1},2j-1)$ and $(r_j,2j+1)$ lie in the same row, so by rule~\textup{(\texttt{d})} the cell \((r_j,2j)\) becomes
contaminated. Hence every cell of \(D(S)\) belongs to \(C(S)\).

It remains to show that no cell outside \(D(S)\) can be contaminated from
\(D(S)\). For this, we record three immediate observations.

\noindent\textit{(i) Each column contains at most one cell of \(D(S)\).}
Indeed, each odd column contains exactly one cell of \(S\), and an even column
contains at most one cell, namely \((r_j,2j)\) when \(r_{j-1}=r_j\).

\noindent\textit{(ii) If two consecutive columns are both nonempty in \(D(S)\),
then their contaminated cells lie in the same row.}
Indeed, if column \(2j\) is nonempty, then by definition its unique cell is
\((r_j,2j)\), while column \(2j+1\) contains \((r_j,2j+1)\). The same argument
applies to the pair \((2j-1,2j)\).

\noindent\textit{(iii) If two columns at distance \(2\) are both nonempty in
\(D(S)\), then their contaminated cells lie either in the same row or in two
consecutive rows.}
If the two columns are odd, say \(2j-1\) and \(2j+1\), then their contaminated
cells are \((r_{j-1},2j-1)\) and \((r_j,2j+1)\). Since \(w\) avoids \(13\) and
\(31\), we have \(|r_j-r_{j-1}|\le 1\). If the two columns are even, say
\(2j\) and \(2j+2\), then both being nonempty implies $r_{j-1}=r_j=r_{j+1}$,  so the two contaminated cells lie in the same row.

We now check that none of the rules \textup{(\texttt{a})}--\textup{(\texttt{h})} can produce a
cell outside \(D(S)\).

Rules \textup{(\texttt{a})} and \textup{(\texttt{b})} require two contaminated cells in columns
at distance \(2\) and in rows \(1\) and \(3\), which is impossible by
observation~(iii). Rule \textup{(\texttt{c})} requires two contaminated cells in the
same column, which is impossible by observation~(i). Rules
\textup{(\texttt{e})}--\textup{(\texttt{h})} require two contaminated cells in consecutive
columns and adjacent rows, which is impossible by observation~(ii).

Finally, consider rule~\textup{(\texttt{d})}. If it applies to two contaminated cells in
the same row and in columns at distance \(2\), then either the target cell lies
in an even column \(2j\), in which case it is exactly \((r_j,2j)\in D(S)\), or
the target cell lies in an odd column, which already contains its unique
initial contaminated cell from \(S\subseteq D(S)\).

Therefore no cell outside \(D(S)\) can be contaminated from \(D(S)\). Since
\(D(S)\subseteq C(S)\), it follows that \(C(S)=D(S)\).

By observation~(i), each column of \(C(S)\) contains at most one contaminated
cell. Hence \(C(S)\neq V(G(3,2k+1))\), so \(S\) does not fully contaminate the
grid.\qedhere
\end{proof}

\begin{lemma}\label{lem:3row-factor}
Let \(S\) be an optimal contamination set of \(G(3,2k+1)\), and let
\(w=r_0r_1\cdots r_k\in\{1,2,3\}^{k+1}\) be its associated word. Then \(S\)
fully contaminates \(G(3,2k+1)\) if and only if \(w\) contains the factor
\(13\) or \(31\).
\end{lemma}

\begin{proof}
Suppose first that \(w\) contains the factor \(13\) or \(31\). Then for some
\(t\in\{0,\dots,k-1\}\), the contaminated cells in columns \(2t+1\) and
\(2t+3\) lie in rows \(1\) and \(3\), in some order. Hence, inside the
\(3\times 3\) subgrid induced by columns \(2t+1,2t+2,2t+3\), these two cells
are positioned according to rule~\textup{(\texttt{a})} or rule~\textup{(\texttt{b})}. By Observation~\ref{obse:observations}~(A), this \(3\times 3\) subgrid becomes fully
contaminated.

From there, the contamination propagates to the right as follows. Assume that
three consecutive columns \(2j+1,2j+2,2j+3\) are fully contaminated, with
\(2j+5\le 2k+1\). Since column \(2j+5\) already contains its initial seed, rule
\textup{(\texttt{d})} first contaminates one cell in column \(2j+4\); then, using the
fully contaminated column \(2j+3\), rules \textup{(\texttt{e})}--\textup{(\texttt{h})} fill
column \(2j+4\), after which column \(2j+5\) becomes fully contaminated as
well. Repeating this argument propagates contamination all the way to the
right. By symmetry, the same argument propagates contamination to the left.
Hence the whole grid becomes contaminated.

Conversely, suppose that \(w\) avoids both \(13\) and \(31\). Then
Lemma~\ref{lem:3row-avoid} gives \(C(S)=D(S)\), and in particular each column
of \(C(S)\) contains at most one contaminated cell. Therefore
\(C(S)\neq V(G(3,2k+1))\), so \(S\) does not fully contaminate the grid.

Thus \(S\) fully contaminates \(G(3,2k+1)\) if and only if \(w\) contains the
factor \(13\) or \(31\).\qedhere
\end{proof}

\begin{proposition}\label{prop:alpha-3-odd}
For every integer \(k\ge 1\), $\alpha_{3,2k+1}$ counts the number of ternary words of length \(k+1\) over \(\{1,2,3\}\)
containing at least one occurrence of \(13\) or \(31\). This number is given by \textup{\texttt{A193519}}$(k+1)$.
\end{proposition}

\begin{proof}
By Lemma~\ref{lem:canonical-odd-columns}, every optimal solution on
\(G(3,2k+1)\) is encoded by a unique word $w=r_0r_1\cdots r_k\in\{1,2,3\}^{k+1}$.

By Lemma~\ref{lem:3row-factor}, such a configuration fully contaminates the grid
if and only if the associated word contains the factor \(13\) or \(31\).
Therefore \(\alpha_{3,2k+1}\) is exactly the number of ternary words of length
\(k+1\) containing \(13\) or \(31\).\qedhere
\end{proof}

\begin{remark}
Equivalently, $\alpha_{3, 2k+1}=3^{k+1}-u_{k+1}$, where $u_{n}$ counts ternary words of length $n$ over $\{1,2,3\}$ avoiding both $13$ and $31$.
\end{remark}

\subsection{Odd square grids, permutations, and large Schr\"oder numbers}

In what follows, we discuss the connection between the quantity 
$\alpha_{2k+1, 2k+1}$ and the OEIS sequence $\texttt{A006318}$, 
which enumerates the $k$-th \emph{large Schröder numbers}. First, we introduce some preliminaries that shall be used throughout this discussion.
\begin{definition}
    A permutation in $\mathcal{S}_{k}$ is an arrangement of the numbers $1, 2, \ldots, k$ for some $k \geq 1$.  
A permutation $\pi$ is \emph{contained} in another permutation $\sigma$ if $\sigma$ has a subsequence whose terms are in the same relative order as those of $\pi$. Otherwise, we say that $\sigma$ \emph{avoids} $\pi$.
\end{definition}

\begin{example} The permutation $3142$ is contained in $1573462$ because the subsequence $5362$ is ordered the same way as $3142$.  
On the other hand, if we take $\pi = 2413 \in \mathcal{S}_4$ and consider $\sigma_{1} = 43125$ and $\sigma_{2} = 35124$, we can check containment by listing all subsequences of length $4$, rewriting them as elements of $\mathcal{S}_4$, and comparing to $\pi$.

\begin{table}[H]
\centering
\begin{tabular}{cc| cc}
\hline
\textit{subsequences of $\sigma_1$} & \textit{image in $S_4$} &
\textit{subsequences of $\sigma_2$} & \textit{image in $S_4$} \\
\hline
4312 & 4312 & 3512 & 3412 \\
4315 & 3214 & 3514 & \textbf{2413} \\
4325 & 3214 & 3524 & \textbf{2413} \\
4125 & 3124 & 3124 & 3124 \\
3125 & 3124 & 5124 & 4123 \\
\hline
\end{tabular}
\caption{Subsequences of $\sigma_1$ and $\sigma_2$ and their images in $S_4$.}
\label{tab:subsequences-images}
\end{table}

As we can see in Table~\ref{tab:subsequences-images}, none of the subsequences of length $4$ of $\sigma_1$ matches $\pi$.  
On the other hand, two of the subsequences of length $4$ of $\sigma_2$ match $\pi$, shown in bold.  
Therefore, $\sigma_2$ contains $\pi$.

\end{example}

\begin{proposition}\label{prop:odd-square-permutations}
For every integer $k\ge 1$, each optimal solution of $G(2k+1,2k+1)$ corresponds to a permutation
of the set $\{1,3,\dots,2k+1\}$.
 
More precisely, if $j$ is odd, and $i_j$ denotes the unique row occupied by the contaminated cell
in column $j$, then $i_1i_3\cdots i_{2k+1}$ is a permutation of $\{1,3,\dots,2k+1\}$.
\end{proposition}

\begin{proof}
Let $S$ be an optimal solution of $G(2k+1,2k+1)$. Since the grid has odd width and $2k+1\ge 3$, Lemma~\ref{lem:canonical-odd-columns} applies. Hence each odd column
contains exactly one contaminated cell of $S$, and each even column contains no contaminated
cell.

Because $G(2k+1,2k+1)$ is a square grid, the same argument may be applied by symmetry to
the rows: each odd row contains exactly one contaminated cell of $S$, and each even row
contains none.

Now, for each odd column $j\in\{1,3,\dots,2k+1\}$, let $i_j$ denote the unique row occupied
by the contaminated cell in column $j$. Since every odd row contains exactly one contaminated
cell, the values $i_1,i_3,\dots,i_{2k+1}$ are precisely the odd integers $1,3,\dots,2k+1$,  each appearing exactly once. Therefore, $i_1 i_3 \cdots i_{2k+1}$ is a permutation of the set $\{1,3,\dots,2k+1\}$.\qedhere
\end{proof}

\begin{conjecture}\label{conj:schroder}
    
For all integers $k \geq 1$, we have
\begin{equation}
    \alpha_{2k+1, 2k+1} = \textup{\texttt{A006318}}(k),
\end{equation}
where \textup{\texttt{A006318}}$(k)$ denotes the $k$-th \emph{Large Schröder number}.
\end{conjecture}

By Proposition~\ref{prop:odd-square-permutations}, every optimal solution of $G(2k+1,2k+1)$
corresponds to a permutation of the set $\{1,3,\dots,2k+1\}$.
 
Equivalently, after relabeling $1,3,\dots,2k+1$ by $1,2,\dots,k+1$, each optimal solution gives
rise to a permutation in $S_{k+1}$.

For example, in $G(7,7)$ we have $\gamma_c(G(7,7))=4$, and Figure~\ref{fig:solutions_of_7_by_7} illustrates all $22$
optimal solutions, while Table~\ref{tab:sol_of_7_7} lists the corresponding permutations together with their images
in $S_4$. Among the $4!=24$ permutations, the two missing ones are $5173$ and $3715$, whose
images in $S_4$ are precisely the forbidden patterns $3142$ and $2413$, respectively.

Similarly, for $G(9,9)$, out of the $5!=120$ permutations of $\{1,3,5,7,9\}$, exactly $90$
correspond to optimal solutions; the remaining $30$ contain one of the patterns $2413$ or $3142$.

These observations strongly suggest that the optimal solutions of $G(2k+1,2k+1)$ are governed
by the same pattern-avoidance phenomenon that characterizes the large Schröder numbers.
Accordingly, Conjecture~\ref{conj:schroder} may be viewed as a pattern-avoidance conjecture for
optimal contamination sets on odd square grids. A proof would require a direct characterization
of which permutations of $\{1,3,\dots,2k+1\}$ correspond to optimal contamination sets, ideally
in terms of avoidance of the patterns $2413$ and $3142$.
 
  \subsection{General bounds and extension mechanisms}
 
We conclude the section with more general estimates and construction principles for odd-width grids. These results show how optimal solutions may be generated, bounded, or extended, and provide further evidence for the rich combinatorial structure underlying the problem.

\begin{figure}[H]
    \centering



    \caption{All optimal solutions of $G(7,7)$.}
    \label{fig:solutions_of_7_by_7}
\end{figure}

\begin{table}[H]
\centering
\resizebox{\columnwidth}{!}{
\begin{tabular}{@{}ccc|ccc|ccc@{}}
\toprule
Solution & Permutation & Image in $\mathcal{S}_{4}$ &
Solution & Permutation & Image in $\mathcal{S}_{4}$ &
Solution & Permutation & Image in $\mathcal{S}_{4}$ \\ \midrule
(1) & 1357 & 1234 & (9)  & 5137 & 3124 & (17) & 3571 & 2341 \\
(2) & 1375 & 1243 & (10) & 7135 & 4123 & (18) & 3751 & 2431 \\
(3) & 1537 & 1324 & (11) & 7153 & 4132 & (19) & 5371 & 3241 \\
(4) & 1735 & 1423 & (12) & 3517 & 2314 & (20) & 7351 & 4231 \\
(5) & 1573 & 1342 & (13) & 5317 & 3214 & (21) & 5731 & 3421 \\
(6) & 1753 & 1432 & (14) & 7315 & 4213 &  (22)     &  7531    &  4321      \\
(7) & 3157 & 2134 & (15) & 5713 & 3412 &       &      &       \\
(8) & 3175 & 2143 & (16) & 7513 & 4312 &       &      &       \\ \bottomrule
\end{tabular}}
\caption{Permutations of $G(7,7)$ optimal solutions and their images in $\mathcal{S}_4$.}
\label{tab:sol_of_7_7}
\end{table}

\begin{proposition}\label{prop:odd-m-upper-bound}
For all integers $m\ge n\ge 3$ with $m$ odd, we have
\begin{equation}
    \alpha_{n,m}\le n^{\gamma_c(G(n,m))}.
\end{equation}
\end{proposition}

\begin{proof}
Let $m=2k+1$. By Lemma~\ref{lem:canonical-odd-columns}, every optimal solution of $G(n,m)$ contains exactly one contaminated cell in each odd column and no contaminated cell in any even column. Hence an optimal solution is completely determined by choosing, independently for each of the $k+1$ odd columns, one of the $n$ rows in which the contaminated cell is placed.

Therefore,
\[
\alpha_{n,m}\le n^{k+1}.
\]
Since
\[
k+1=\frac{m+1}{2}=\left\lfloor \frac m2\right\rfloor+1=\gamma_c(G(n,m)),
\]
the result follows.\qedhere
\end{proof}

\begin{proposition}\label{prop:square-upper-bounds}
For every odd integer $m\ge 3$, we have
\[
\alpha_{m,m}\le \min\!\left\{ m^{\gamma_c(G(m,m))},\ \gamma_c(G(m,m))!\right\}.
\]
\end{proposition}

\begin{proof}
The first bound  follows directly from Proposition~\ref{prop:odd-m-upper-bound} applied with $n=m$.

For the second bound, let $m=2k+1$. By Proposition~\ref{prop:odd-square-permutations}, each
optimal solution of $G(m,m)$ corresponds to a permutation of the set $\{1,3,\dots,2k+1\}$, which has cardinality $k+1$. Therefore,
\[
\alpha_{m,m}\le (k+1)!.
\]
Since
\[
k+1=\left\lfloor\frac{m}{2}\right\rfloor+1=\gamma_c(G(m,m)),
\]
it follows that
\[
\alpha_{m,m}\le \gamma_c(G(m,m))!.
\]
Combining the two bounds, we obtain the result.\qedhere
\end{proof}

\begin{proposition}
    Consider a grid $G(n,m)$ with $m-2 \geq n \geq 3$ and $m$ odd. Then we have the following: 
    \begin{itemize}
        \item Each optimal solution of $G(n,m-2)$ can be used to generate $n$ 
        optimal solutions of $G(n,m)$;
        \item Not all optimal solutions of $G(n,m)$ can be used to generate a solution for $G(n,m-2)$.
    \end{itemize}
\end{proposition}

\begin{proof}
    Consider the first $m-2$ columns of the grid $G(n,m)$, which together correspond to $G(n,m-2)$.  Any optimal solution for $G(n,m-2)$ that uses $\gamma_c(G(n,m-2))$ 
    contaminated cells will fully contaminate the first $m-2$ columns of $G(n,m)$.  
    Since  $\gamma_c(G(n,m)) = \gamma_c(G(n,m-2)) + 1$, the remaining contaminated cell can be chosen in any of the $n$ cells of the last column $m$, yielding a total of $n$ possible optimal solutions for $G(n,m)$.

    However, not all optimal solutions of $G(n,m)$ can be used to generate a solution for $G(n,m-2)$ by simply removing the last two columns.  
    Figure~\ref{fig:solution_for_n3_m5} shows such an example.\qedhere
\end{proof}

\begin{figure}[H]
    \centering

\begin{tikzpicture}[scale=0.8pt,x=0.75pt,y=0.75pt,yscale=-1,xscale=1]
%uncomment if require: \path (0,688); %set diagram left start at 0, and has height of 688

%Shape: Square [id:dp05490031373510984] 
\draw  [fill={rgb, 255:red, 208; green, 2; blue, 27 }  ,fill opacity=1 ] (292.33,27.33) -- (312.33,27.33) -- (312.33,47.33) -- (292.33,47.33) -- cycle ;
%Shape: Square [id:dp9205428908247857] 
\draw  [fill={rgb, 255:red, 255; green, 255; blue, 255 }  ,fill opacity=1 ] (312.33,27.33) -- (332.33,27.33) -- (332.33,47.33) -- (312.33,47.33) -- cycle ;
%Shape: Square [id:dp9658188906521039] 
\draw  [fill={rgb, 255:red, 255; green, 255; blue, 255 }  ,fill opacity=1 ] (312.33,47.33) -- (332.33,47.33) -- (332.33,67.33) -- (312.33,67.33) -- cycle ;
%Shape: Square [id:dp4841030125853112] 
\draw  [fill={rgb, 255:red, 255; green, 255; blue, 255 }  ,fill opacity=1 ] (292.33,47.33) -- (312.33,47.33) -- (312.33,67.33) -- (292.33,67.33) -- cycle ;
%Shape: Square [id:dp6836892533315282] 
\draw  [fill={rgb, 255:red, 208; green, 2; blue, 27 }  ,fill opacity=1 ] (332.33,27.33) -- (352.33,27.33) -- (352.33,47.33) -- (332.33,47.33) -- cycle ;
%Shape: Square [id:dp11963662984002532] 
\draw  [fill={rgb, 255:red, 255; green, 255; blue, 255 }  ,fill opacity=1 ] (332.33,47.33) -- (352.33,47.33) -- (352.33,67.33) -- (332.33,67.33) -- cycle ;
%Shape: Square [id:dp1105012511531267] 
\draw  [fill={rgb, 255:red, 255; green, 255; blue, 255 }  ,fill opacity=1 ] (352.33,47.33) -- (372.33,47.33) -- (372.33,67.33) -- (352.33,67.33) -- cycle ;
%Shape: Square [id:dp22432377725641417] 
\draw  [fill={rgb, 255:red, 255; green, 255; blue, 255 }  ,fill opacity=1 ] (352.33,27.33) -- (372.33,27.33) -- (372.33,47.33) -- (352.33,47.33) -- cycle ;
%Shape: Square [id:dp6607934888694917] 
\draw  [fill={rgb, 255:red, 255; green, 255; blue, 255 }  ,fill opacity=1 ] (372.33,47.33) -- (392.33,47.33) -- (392.33,67.33) -- (372.33,67.33) -- cycle ;
%Shape: Square [id:dp2690765655842198] 
\draw  [fill={rgb, 255:red, 255; green, 255; blue, 255 }  ,fill opacity=1 ] (372.33,27.33) -- (392.33,27.33) -- (392.33,47.33) -- (372.33,47.33) -- cycle ;
%Shape: Square [id:dp18288964269938157] 
\draw  [fill={rgb, 255:red, 255; green, 255; blue, 255 }  ,fill opacity=1 ] (292.33,67.33) -- (312.33,67.33) -- (312.33,87.33) -- (292.33,87.33) -- cycle ;
%Shape: Square [id:dp6823625296492017] 
\draw  [fill={rgb, 255:red, 255; green, 255; blue, 255 }  ,fill opacity=1 ] (312.33,67.33) -- (332.33,67.33) -- (332.33,87.33) -- (312.33,87.33) -- cycle ;
%Shape: Square [id:dp21784412474830495] 
\draw  [fill={rgb, 255:red, 255; green, 255; blue, 255 }  ,fill opacity=1 ] (332.33,67.33) -- (352.33,67.33) -- (352.33,87.33) -- (332.33,87.33) -- cycle ;
%Shape: Square [id:dp49585271300653555] 
\draw  [fill={rgb, 255:red, 255; green, 255; blue, 255 }  ,fill opacity=1 ] (352.33,67.33) -- (372.33,67.33) -- (372.33,87.33) -- (352.33,87.33) -- cycle ;
%Shape: Square [id:dp9061487812899023] 
\draw  [fill={rgb, 255:red, 208; green, 2; blue, 27 }  ,fill opacity=1 ] (372.33,67.33) -- (392.33,67.33) -- (392.33,87.33) -- (372.33,87.33) -- cycle ;

\end{tikzpicture}
\caption{An optimal solution for $G(3,5)$ that cannot be used to generate a solution for $G(3,3)$ by removing the last two columns.}
    \label{fig:solution_for_n3_m5}
\end{figure}

\section{Conclusion and Perspectives}\label{sec:conclusion}

In this paper, we gave a complete solution to the power contamination problem on rectangular grid graphs. We first disproved Conjecture~\ref{conj}   by exhibiting a counterexample, and then determined the exact value of the contamination number \(\gamma_c(G(n,m))\) for all integers \(m \ge n \ge 1\). In addition, we derived recurrence relations and structural properties of optimal contamination sets, especially in the odd-width case. We then initiated a combinatorial study of the number \(\alpha_{n,m}\) of optimal solutions. For several grid families, we obtained explicit formulas and interpretations in terms of classical combinatorial objects, including binary choices, ternary words with forbidden factors, and permutations associated with odd square grids. These results reveal that the power contamination problem on grids is not only an extremal propagation problem but also a rich source of enumerative phenomena, with connections to well-known integer sequences and, conjecturally, to the large Schr\"oder numbers.

Several directions remain open for future research. On the enumerative side, it would be especially interesting to settle Conjecture~\ref{conj:alpha2odd-total}, which predicts a closed formula for \(\alpha_{2,2k+1}\), and Conjecture~\ref{conj:schroder}, which suggests that the numbers \(\alpha_{2k+1,2k+1}\) are given by the large Schr\"oder numbers and may admit a description in terms of permutation pattern avoidance.   Another natural direction is to study the contamination speed, namely the number of propagation steps needed to contaminate the whole grid from an optimal set. It would also be interesting to extend the analysis to other graph classes or lattice structures, and to investigate algorithmic aspects of the problem, including exact methods, constructive generation of optimal solutions, and efficient recognition or counting procedures for special families of graphs.

 \section*{Statements and Declarations}

\subsection*{Funding}
The authors declare that no funds, grants, or other support were received during the preparation of this manuscript.

\subsection*{Competing Interests}
The authors have no relevant financial or non-financial interests to disclose.

\subsection*{Author Contributions}

Both authors contributed equally to this work. Both authors have read and approved the final manuscript.

\subsection*{Data Availability}
No data were used in this study.

\bibliographystyle{plainurl} 
\bibliography{bibou}

\end{document}